\documentclass{amsart}

\usepackage{amssymb,hyperref,enumitem,mathtools,mathrsfs,graphicx,tikz}
\usepackage[alphabetic,initials,nobysame]{amsrefs}

\usetikzlibrary{patterns,decorations.pathreplacing}

\usepackage{color,comment,todonotes}   

\BibSpec{collection.article}{%
	+{}  {\PrintAuthors}				{author}
	+{,} { \textit}                     {title}
	+{.} { }                            {part}
	+{:} { \textit}                     {subtitle}
	+{,} { \PrintContributions}         {contribution}
	+{,} { \PrintConference}            {conference}
	+{}  {\PrintBook}                   {book}
	+{,} { }                            {booktitle}
	+{,} { }							{series}
    +{,} { vol.~}                       {volume}
	+{,} { }							{publisher}
	+{,} { }							{address}
	+{,} { \PrintDateB}                 {date}
	+{,} { pp.~}                        {pages}
	+{,} { }                            {status}
	+{,} { \PrintDOI}                   {doi}
	+{,} { available at \eprint}        {eprint}
	+{}  { \parenthesize}               {language}
	+{}  { \PrintTranslation}           {translation}
	+{;} { \PrintReprint}               {reprint}
	+{.} { }                            {note}
	+{.} {}                             {transition}
	+{}  {\SentenceSpace \PrintReviews} {review}
}

\newtheorem{theorem}{Theorem}[section]
\newtheorem{lemma}[theorem]{Lemma}
\newtheorem{corollary}[theorem]{Corollary}
\newtheorem{proposition}[theorem]{Proposition}

\theoremstyle{remark}
\newtheorem{remark}[theorem]{Remark}

\theoremstyle{definition}
\newtheorem{definition}[theorem]{Definition}

\newcommand{\R}{\mathbb{R}}
\newcommand{\C}{\mathbb{C}}

\newcommand{\N}{\mathbb{N}}
\newcommand{\D}{\mathbb{D}}
\newcommand{\loc}{\mathrm{loc}}
\renewcommand{\mod}{\operatorname{Mod}}
\renewcommand{\bar}[1]{\overline{#1}}

\DeclareMathOperator{\dist}{dist}
\DeclareMathOperator{\inter}{int}
\DeclareMathOperator{\diam}{diam}
\renewcommand{\Re}{\operatorname{Re}}
\renewcommand{\Im}{\operatorname{Im}}

\makeatletter
\DeclareFontFamily{U}{tipa}{}
\DeclareFontShape{U}{tipa}{m}{n}{<->tipa10}{}
\newcommand{\arc@char}{{\usefont{U}{tipa}{m}{n}\symbol{62}}}%

\newcommand{\arc}[1]{\mathpalette\arc@arc{#1}}
\newcommand{\arc@arc}[2]{%
  \sbox0{$\m@th#1#2$}%
  \vbox{
    \hbox{\resizebox{\wd0}{\height}{\arc@char}}
    \nointerlineskip
    \box0
  }%
}
\makeatother

\numberwithin{equation}{section}
\numberwithin{figure}{section}

\title{Piecewise quasiconformal dynamical systems of the unit circle}

\author{Yusheng Luo}
\address[Yusheng Luo]{Department of Mathematics, Cornell University, 212 Garden Ave, Ithaca, NY 14853, USA}
\email{yusheng.s.luo@gmail.com}
\thanks{The first-named author is partially supported by NSF Grant DMS-2349929.}

\author{Dimitrios Ntalampekos}
\address[Dimitrios Ntalampekos]{Department of Mathematics, Aristotle University of Thessaloniki, Thessaloniki, 54152, Greece.}
\email[Corresponding author]{dntalam@math.auth.gr}

\thanks{The second-named author was partially supported by NSF Grant DMS-2246485}

\begin{document}
\begin{abstract}
    We study piecewise quasiconformal covering maps of the unit circle. We provide sufficient conditions so that a conjugacy between two such dynamical systems has a quasiconformal or David extension to the unit disk. Our main result generalizes the main result of \cite{LyubichMerenkovMukherjeeNtalampekos:David}, which deals with piecewise analytic maps. As applications, we provide a classification of piecewise quasiconformal maps of the circle up to quasisymmetric conjugacy, we prove a general conformal mating theorem for Blaschke products, and we study the quasiconformal geometry of parabolic basins. 
\end{abstract}
\maketitle

\setcounter{tocdepth}{1}
\tableofcontents

\section{Introduction}

The purpose of the paper is to study piecewise quasiconformal covering maps of the unit circle. The main theorem provides sufficient conditions so that a homeomorphism conjugating two such maps admits a quasiconformal or David extension to the unit disk. Due to the technicalities of the results, we give rough formulations in the introduction and defer the full statements to later sections.

\subsection{Main extension theorem}
Let $f\colon \mathbb S^1\to \mathbb S^1$ be a covering map and let $\{a_0,\dots,a_r\}$ be a Markov partition associated to $f$. For $k\in \{0,\dots,r\}$ we define $A_k$ to be the closed arc from $a_k$ to $a_{k+1}$, where $a_{r+1}=a_0$. We consider three basic assumptions that will be described in more detail in Sections \ref{section:markov_expansive} and \ref{section:extension}.  
\begin{enumerate}[label=\normalfont(\arabic*)]
\medskip
    \item (Endpoint behavior) Each point $a\in \{a_0,\dots,a_r\}$ is either symmetrically hyperbolic or symmetrically parabolic, as defined in Section \ref{section:hyperbolic_parabolic_points}. See condition \ref{condition:hp}.
\medskip
    \item (Extension to neighborhood of the circle) For each $k\in \{0,\dots,r\}$ there exist open neighborhoods $U_k$ of $\inter{A_k}$ and $V_k$ of $f(\inter{A_k})$ in $\C$ such that $f$ has an extension to a {homeomorphism} from $U_k$ onto $V_k$. Furthermore, we assume that $U_j\subset V_k=f(U_k)$ whenever $A_j\subset f(A_k)$ and that the regions $U_k$, $k\in \{0,\dots,r\}$, are pairwise disjoint; see Figure \ref{figure:condition2}. We still denote the extension to $\bigcup_{k=0}^r (U_k\cup \{a_k,a_{k+1}\})$ by $f$. See condition \ref{condition:uv}. 
\medskip
    \item (Quasiconformality) The iterates of $f$ on preimages of $U_k$, $k\in \{0,\dots,r\}$, are uniformly quasiconformal. See condition \ref{condition:qs}.
\medskip
\end{enumerate}

\begin{figure}
    \centering
    \begin{tikzpicture}
        \node {\includegraphics[scale=0.4]{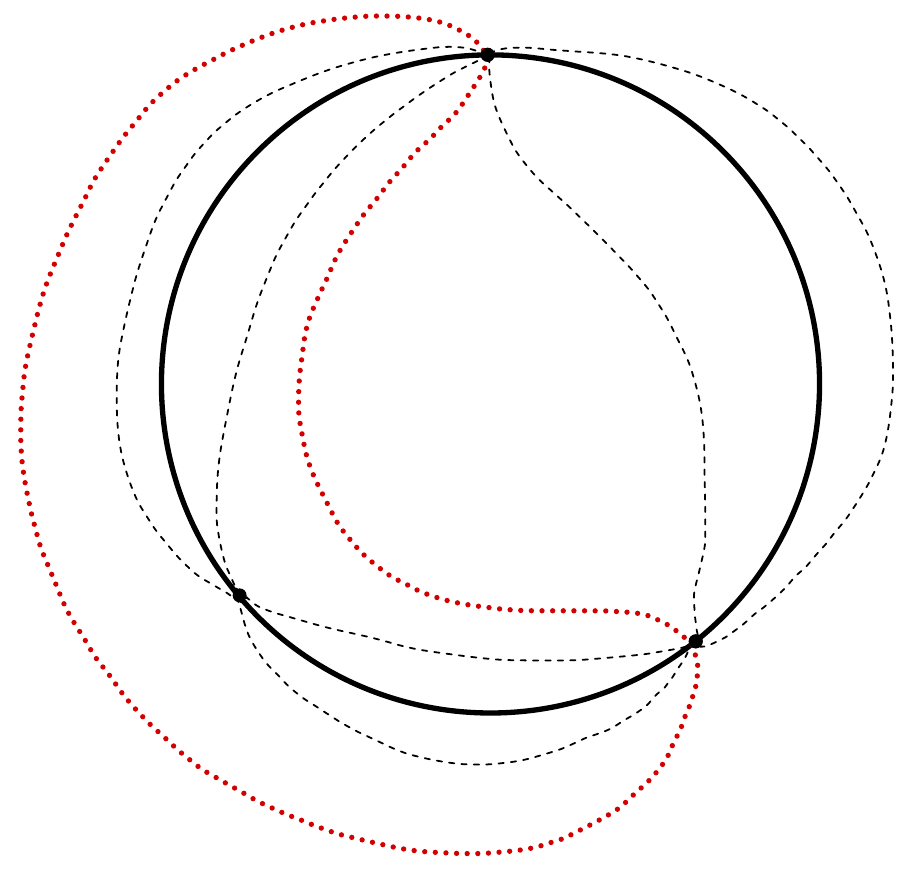}};
        \node at (1.7,1.7) {$U_0$};
        \node at (-1.3,1.6) {$U_1$};
        \node at (0, -1.7) {$U_2$};
        \node at (0.3,2.8) {$a_1$};
        \node at (-1.7, -1.3) {$a_2$};
        \node at (1.9, -1.7) {$a_0$};
        \node at (-3.2, 0) {$V_0$};
    \end{tikzpicture}
    
    \caption{Illustration of condition (2).}
    \label{figure:condition2}
\end{figure}
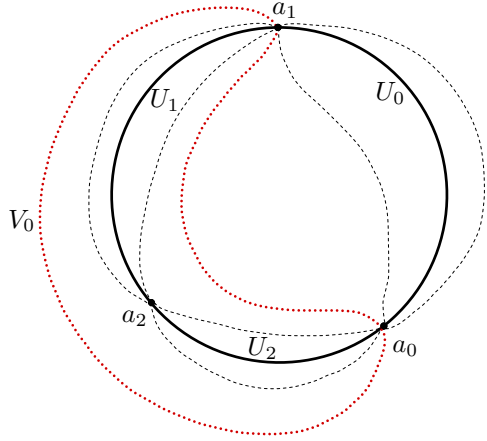

We now state the main results. See Theorem \ref{theorem:extension_generalization} for more detailed statements.

\begin{theorem}[QC extension]\label{theorem:qc_extension}
    Let $f,g\colon \mathbb S^1\to \mathbb S^1$  be expansive covering maps with the same orientation and $\{a_0,\dots,a_r\}$, $\{b_0,\dots,b_r\}$ be Markov partitions associated to $f,g$, respectively, satisfying conditions \textup{(1)}, \textup{(2)}, and \textup{(3)}. Let $h\colon \mathbb S^1\to \mathbb S^1$ be a homeomorphism that conjugates $f$ to $g$ such that, for each $k\in \{0,\dots,r\}$, $h(a_k)=b_k$ and 
    \begin{enumerate}[label= {\textup{\scriptsize(\textbf{H/P${\rightarrow}$H/P})}}, leftmargin=6em] 
        \item if $a_k$ is symmetrically hyperbolic (resp.\ parabolic), then $b_k$ is symmetrically hyperbolic (resp.\ parabolic).
    \end{enumerate}
    Then the map $h$ has a quasiconformal extension to the unit disk.
\end{theorem}

Condition {\normalfont\scriptsize(\textbf{H/P${\rightarrow}$H/P})} states that the conjugacy $h$ does not alter the hyperbolic or parabolic nature of points. If we allow the possibility that hyperbolic points are mapped to parabolic ones, then we cannot obtain a quasiconformal extension to the unit disk. Instead, in that case we prove that the conjugacy $h$ has an extension to a David homeomorphism (also known as mapping of exponentially integrable distortion). To prove this result, we need to add a modification of condition (3):
\begin{enumerate}
    \item[(3*)] (Asymptotic conformality) The iterates of $f$ on preimages of $U_k$, $k\in \{0,\dots,r\}$, are $(1+o(1))$-quasiconformal. See condition \ref{condition:qs_strong}.
\end{enumerate}
This condition essentially compensates for the unavailability of Koebe's distortion theorem in the theory of quasiconformal maps. 

\begin{theorem}[David extension]\label{theorem:david_extension}
    Let $f,g,h$ be as in Theorem \ref{theorem:qc_extension} and suppose that $g$ satisfies condition \textup{(3*)}. Suppose that each pair $a_k,b_k$, $k\in \{0,\dots,r\}$, satisfies either condition {\normalfont\scriptsize(\textbf{H/P${\rightarrow}$H/P})} or
\begin{enumerate}[label={\textup{\scriptsize(\textbf{H$\to$P})}}, leftmargin=6em]
\item $a_k$ is symmetrically hyperbolic and $b_k$ is symmetrically parabolic.
\end{enumerate}
Then $h$ has a David extension to the unit disk. 
\end{theorem}

These two results generalize the main result of \cite{LyubichMerenkovMukherjeeNtalampekos:David} in two directions. First, in \cite{LyubichMerenkovMukherjeeNtalampekos:David} the maps $f,g$ are assumed to be piecewise analytic rather than piecewise quasiconformal. Second, the definition of symmetrically hyperbolic and parabolic points is much less restrictive here (see Section \ref{section:hyperbolic_parabolic_points}) compared to  \cite{LyubichMerenkovMukherjeeNtalampekos:David}. There, $f|_{A_k}$ is required to have an analytic extension in neighborhoods of the endpoints $a_k$ and $a_{k+1}$, so hyperbolicity and parabolicity are defined in the natural way, by studying the power series expansion of $f$. Here we do not assume any extension to neighborhoods of the endpoints, but instead we distinguish between hyperbolic and parabolic points based on how $f$ distorts lengths of circular arcs near the endpoints.

The main motivation for extending the results of \cite{LyubichMerenkovMukherjeeNtalampekos:David} to the piecewise quasiconformal setting is due to an application in the problem of uniformization of gasket Julia sets. In \cite{LN24b} we use the present results to provide a characterization of gasket Julia sets that can be uniformized by a round gasket using a quasiconformal or David homeomorphism of the sphere.

Using Theorem \ref{theorem:qc_extension} we obtain a classification of piecewise quasiconformal circle maps up to quasisymmetries. See Theorem \ref{theorem:classification} and Remark \ref{remark:classification}. We formulate a simplified version here. 

\begin{theorem}[QS classification]\label{theorem:intro_classification}
    Let $f\colon \mathbb S^1\to \mathbb S^1$ be an expansive covering map with a sufficiently fine Markov partition satisfying conditions \textup{(1)}, \textup{(2)}, and \textup{(3)}. Then $f$ is quasisymmetrically conjugate to a piecewise (anti-)M\"obius map of the unit circle. 
\end{theorem}

\subsection{Blaschke products and mating}
We remark that, in practice, condition (2) is not always easy to verify and the its validity depends on the choice of the Markov partition. Let $B$ be a Blaschke product whose Julia set is the unit circle. We say that $B$ is \textit{hyperbolic} if it has an attracting fixed point in $\D$. Otherwise, $\D$ is the basin of a parabolic fixed point and we say that $B$ is \textit{parabolic}. In \cite{LyubichMerenkovMukherjeeNtalampekos:David}*{Example 4.2}, condition (2) is verified  for a specific Markov partition associated to a certain Blaschke product. We verify condition (2) for essentially every Markov partition associated to any Blaschke product.
\begin{theorem}\label{theorem:blaschke_conditions}
    Let $f\colon \mathbb S^1 \to \mathbb S^1$ be a parabolic (resp.\ hyperbolic) Blaschke product, and $\{a_0,\dots,a_r\}$ be a Markov partition associated to $f$ that contains the parabolic fixed point. Then conditions \textup{(1)}, \textup{(2)}, \textup{(3)}, and \textup{(3*)} are satisfied.
\end{theorem}

We use this in combination with Theorem \ref{theorem:david_extension} to show the next mating result, which can also be obtained using deformation techniques in \cites{HT04, Luo22, Luo23}.

\begin{theorem}[Blaschke mating]\label{theorem:intro_mating}
    Let $f,g$ be hyperbolic or parabolic Blaschke products of the same degree whose Julia set is the unit circle. Then $f$ and $g$ are conformally mateable along any conjugacy $h$ from $f|_{\mathbb S^1}$ to $g|_{\mathbb S^1}$.  
    
    \smallskip
    \noindent
    Specifically, there exist a Jordan curve $J$ with complementary regions $A,B\subset \widehat\C\setminus J$, a rational map $R$, and conformal maps $\phi\colon \overline \D\to \overline A$, $\psi\colon \overline \D\to \overline B$ such that $\phi$ conjugates $f$ to $R$, $\psi$ conjugates $g$ to $R$, and $\phi=\psi\circ h$ on $\mathbb S^1$. Moreover, $J$ is a David circle and $R$ is unique up to M\"obius conjugacies.
\end{theorem}

An instance of Theorem \ref{theorem:intro_mating} was proved in \cite{LyubichMerenkovMukherjeeNtalampekos:David}*{Example 5.3}, where the Blaschke product $B(z)=\frac{2z^3+1}{z^3+2}$ was mated with itself so that the parabolic point $1$ is mated with the repelling point $-1$ and the repelling point $-1$ is mated with the parabolic point $1$. The Julia set of the resulting rational map is shown in Figure \ref{figure:pine}.

It is elementary to show that the converse of Theorem \ref{theorem:intro_mating} holds. Namely, if $R$ is a rational map whose Julia set is a Jordan curve, then the second iterate $R^{\circ 2}$ arises as the mating, in the sense of Theorem \ref{theorem:intro_mating}, of two Blaschke products of the same degree whose Julia set is the unit circle. Thus, this observation and Theorem \ref{theorem:intro_mating} provide a complete description of Julia sets that are Jordan curves. Moreover, since David circles are conformally removable (in the sense of Theorem \ref{theorem:removable} below) we obtain the following corollary.
\begin{corollary}\label{corollary:jordan_removable}
    All Jordan curve Julia sets are conformally removable. 
\end{corollary}

\begin{figure}
    \centering
    \includegraphics[scale=0.3]{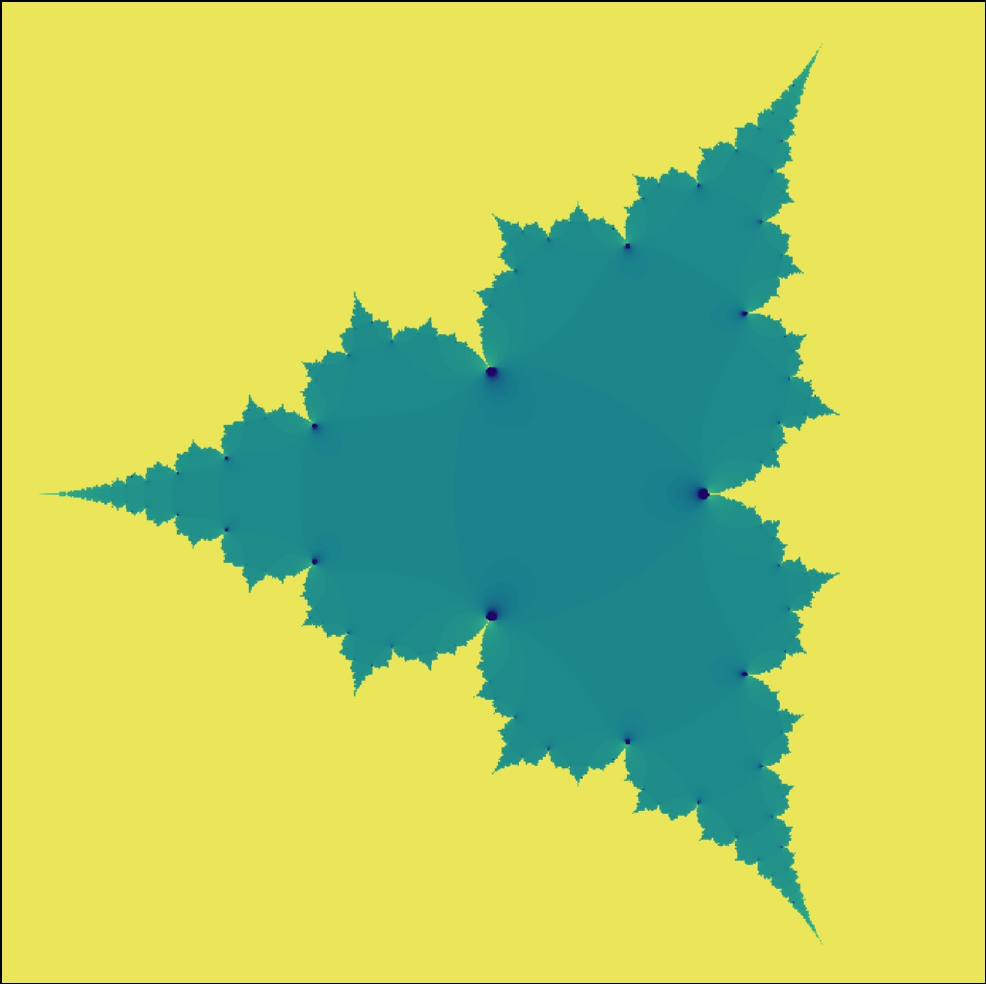}
    \caption{The pine tree Julia set, resulting from mating $B(z)=\frac{2z^3+1}{z^3+2}$ with itself. Cusps are generated because parabolic points are mated with repelling points.}
    \label{figure:pine}
\end{figure}

\subsection{Geometry of parabolic basins}
The Jordan curve $J$ in Theorem \ref{theorem:intro_mating} is the image of the unit circle under a David homeomorphism $H$ of the sphere. If $f$ is hyperbolic and $g$ is parabolic, then the rational map $R$ has a parabolic fixed point with multiplicity $2$. Hence, in that case, the David homeomorphism $H$ maps $\D$ onto a parabolic basin of multiplicity $2$. Note that the inverse of a David homeomorphism is not always a David homeomorphism. In Sections \ref{section:cusp} and \ref{section:qdF}, we study the question whether an arbitrary parabolic basin can be mapped to the unit disk with a David or quasiconformal homeomorphism of the sphere. 
The proof is based on our main extension theorem, Theorem \ref{theorem:extension_generalization}.

\begin{theorem}\label{theorem:basins}
    Let $R$ be a rational map, $a$ be a fixed point with $R'(a) = 1$ that has parabolic multiplicity $\nu\geq 2$, and $\Omega$ be an immediate basin of $a$.
    \begin{enumerate}[label=\normalfont(\arabic*)]
        \item\label{theorem:basins:no_david} If $\nu=2$, then there exists no David homeomorphism $\phi$ of the sphere with $\phi(\Omega)=\D$. 
    \end{enumerate}
        Suppose, in addition, that $\Omega$ is a Jordan region and the critical and post-critical sets intersect $\partial \Omega$ only at the point $a$. 
    \begin{enumerate}[label=\normalfont(\arabic*)]\setcounter{enumi}{1}
        \item\label{theorem:basins:david} If $\nu=2$, then there exists a David homeomorphism $\phi$ of the sphere with $\phi(\D)=\Omega$. 
        \item\label{theorem:basins:qc}  If $\nu\geq 3$, then there exists a quasiconformal homeomorphism $\phi$ of the sphere with $\phi(\D)=\Omega$. 
    \end{enumerate}
\end{theorem}

We remark that in the last part  the condition that the post-critical set intersects $\partial \Omega$ only at the parabolic point $a$ is necessary. Consider the example of a geometrically finite polynomial $f(z) = z-\frac{1}{b(3b-2)}(z-1)^3(z-b)^2$, where $b=-1-\sqrt{2/3}$. It has two parabolic fixed points at $a=1$ with multiplicity $3$ and at $b$ with multiplicity $2$. There is a Fatou component $\Omega$ that is an immediate basin of $a$ such that $b\in \partial \Omega$. The component $\Omega$ is not a quasidisk or a John domain as it has an outward cusp at $b$; see Figure \ref{figure:non_qc_disk}. 
We remark that this condition follows from condition (6.2) in \cite{CJY}*{Theorem 6.1}, which proves an analogue of \ref{theorem:basins:qc} for the polynomial case.

\begin{figure}
    \centering
    \includegraphics[scale=0.4]{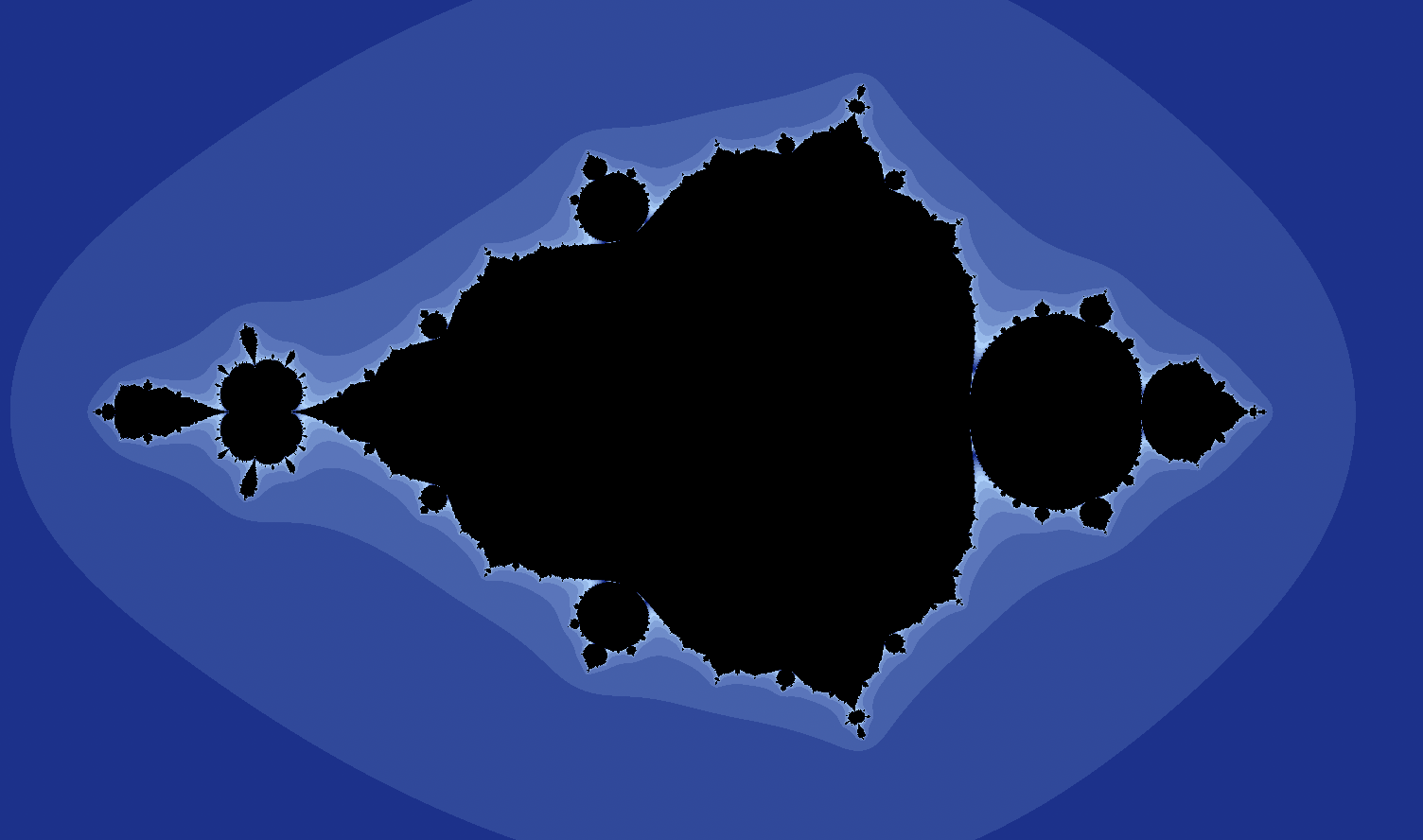}
    \caption{The Julia set of the polynomial $$f(z) = z-\frac{1}{b(3b-2)}(z-1)^3(z-b)^2.$$ The parabolic basin in the center is not a quasidisk.}
    \label{figure:non_qc_disk}
\end{figure}

\section{Preliminaries}

\subsection{Quasiconformal and quasisymmetric maps}

Let $(X,d_X),(Y,d_Y)$ be metric spaces. A homeomorphism $f\colon X\to Y$ is \textit{quasisymmetric} if there exists a homeomorphism $\eta\colon [0,\infty)\to[0,\infty)$ such that for all triples of distinct points $x_1,x_2,x_3\in X$ we have
\begin{align*}
    \frac{d_Y(f(x_1),f(x_2))}{d_Y(f(x_1),f(x_3))} \leq \eta\left( \frac{d_X(x_1,x_2)}{d_X(x_1,x_3)}\right).
\end{align*}
In that case we say that $f$ is $\eta$-quasisymmetric. The homeomorphism $\eta$ is called the \textit{distortion function} associated to $f$. It is an elementary property (see \cite{Heinonen:metric}*{Proposition 10.8}) that if $A\subset B\subset X$ and $0<\diam A\leq \diam B <\infty$, then
\begin{align}\label{quasisymmetry}
\frac{1}{2\eta\left(\frac{\diam B}{\diam A}\right)} \leq    \frac{\diam f(A)}{\diam f(B)}\leq \eta \left(\frac{2\diam A}{\diam B} \right).
\end{align}

Let $U,V\subset \C$ be open sets and $f\colon U\to V$ be an orientation-preserving homeomorphism. We say that $f$ is \textit{quasiconformal} if $f\in W^{1,2}_{\loc}(U)$ and there exists $K\geq 1$ such that $\|Df(z)\|^2 \leq K J_f(z)$ for a.e.\ $z\in U$. Here $Df(z)$ is the differential of $f$, which exists at a.e.\ $z\in U$, $\|Df(z)\|$ denotes its operator norm, and $J_f(z)$ is the Jacobian determinant of $Df(z)$. In that case, we say that $f$ is $K$-quasiconformal. One can define quasiconformal homeomorphisms between domains in the Riemann sphere $\widehat \C$ via local coordinates.   

It is a fundamental result that quasisymmetric maps between planar domains are also quasiconformal, quantitatively. We will need a refined version of the converse of that statement, as stated below. Actually, this is a sharp version of Koebe's distortion theorem for quasiconformal maps.

\begin{theorem}[Distortion theorem]\label{theorem:qcqs}
    Let $K\geq 1$, $\delta>0$, and $U,V\subset \C$ be topological disks with $\bar U\subset V$ and $\mod(V\setminus \bar U)\geq \delta$. If  $f\colon V\to \C$ is a $K$-quasiconformal embedding, then  $f|_{U}$ and $f^{-1}|_{f(U)}$ are $\eta$-quasisymmetric for
    $$\eta(t)= C(K,\delta) \max\{t^K,t^{1/K}\}.$$
\end{theorem}
Here, if $A\subset \C$ is an annulus we denote by $\mod A$ its modulus, i.e., the quantity $(2\pi)^{-1}\log \frac{1}{r}$, where $r\in (0,1)$ is the unique number with the property that $\D\setminus \bar B(0,r)$ is conformally equivalent to $A$.   

\begin{proof}
    Consider a conformal map $\phi\colon  \D \to V$. Since $\mod(V\setminus \bar {U})\geq \delta$, we have $\phi^{-1}(U)\subset B(0,r)$ for some $r=r(\delta)\in (0,1)$, as a consequence of the Gr\"otzsch modulus theorem \cite{LehtoVirtanen:quasiconformal}*{p.~54}. Let $r'=(1+r)/2\in (r,1)$. We now apply \cite{AstalaIwaniecMartin:quasiconformal}*{Lemma 3.6.1}, which implies that there exists a conformal map $\psi$ on $B(0,r')$ and a $K$-quasiconformal map $g\colon \C\to \C$ such that $g(B(0,r'))=B(0,r')$ and $f\circ \phi= \psi\circ g$ on $B(0,r')$.  By Koebe's distortion theorem, the map $\phi|_{B(0,r)}$ satisfies 
    $$C(\delta)^{-1} \frac{|z_1-z_2|}{|z_1-z_3|}\leq \frac{|\phi(z_1)-\phi(z_2)|}{|\phi(z_1)-\phi(z_3)|}\leq C(\delta) \frac{|z_1-z_2|}{|z_1-z_3|}$$
    for all triples of distinct points $z_1,z_2,z_3\in B(0,r)$. The map $\psi|_{B(0,r)}$ also satisfies this condition; see \cite{AstalaIwaniecMartin:quasiconformal}*{Theorem 2.10.9}. Also, by \cite{AstalaIwaniecMartin:quasiconformal}*{Corollary 3.10.4}, $g$ is $\eta''$-quasisymmetric on $\C$ for 
    $$\eta''(t)=C(K) \max\{t^K,t^{1/K}\}.$$
    The desired conclusion for $f= \psi\circ g\circ \phi^{-1}$, restricted to $U$, follows. For $f^{-1}|_{f(U)}$ it suffices to note that if a map is $\eta$-quasisymmetric, then its inverse is $\widetilde \eta$-quasi\-symmetric for $\widetilde \eta(t)=1/\eta^{-1}(t^{-1})$.
\end{proof}
 
\subsection{David maps}
An orientation-preserving homeomorphism $H\colon U\to V$ between domains in the Riemann sphere $\widehat{\C}$ is called a \textit{David homeomorphism} or else \textit{mapping of exponentially integrable distortion}, if it lies in the Sobolev space $W^{1,1}_{\loc}(U)$ (defined via local coordinates) and there exists $p>0$ such that
\begin{align}\label{exp_integral}
\int_U \exp(p K_H) \,d\sigma <\infty,
\end{align}
where $\sigma$ is the spherical measure, and $K_H$ is the distortion function of $H$, given by
\begin{align*}
K_H(z)=\frac{1+|\mu_H|}{1-|\mu_H|};
\end{align*}
here 
\begin{align*}
\mu_H= \frac{\partial H/ \partial \bar z}{\partial H/\partial z}
\end{align*}
is the Beltrami coefficient of $H$. Condition \eqref{exp_integral} is equivalent to the existence of constants $C,\alpha>0$ such that
$$\sigma(\{z\in U: |\mu_H(z)|>1-\varepsilon\}) \leq C e^{-\alpha/\varepsilon}$$
for all $\varepsilon \in (0,1)$. We direct the reader to \cite{AstalaIwaniecMartin:quasiconformal}*{Chapter 20} for more background.  

The main result in the theory of David homeomorphisms is the following integrability theorem. If $U$ is an open subset of $\widehat {\C}$ and $\mu \colon U \to \D$ is a measurable function such that $K=(1+|\mu|)/(1-|\mu|)$ is exponentially integrable in $U$ (i.e., it satisfies \eqref{exp_integral}), then we say that $\mu$ is a \textit{David coefficient (in $U$)}.
\begin{theorem}[David integrability; \cite{AstalaIwaniecMartin:quasiconformal}*{Theorem~20.6.2}]\label{theorem:integrability_david}
Let $\mu\colon \widehat{\C} \to \D$ be a David coefficient.  Then there exists a homeomorphism $H\colon \widehat{\C} \to \widehat{\C}$ of class $ W^{1,1}(\widehat \C)$ that solves the Beltrami equation
\begin{align}\label{beltrami_equation}
\frac{\partial H}{\partial \bar z}= \mu \frac{\partial H}{\partial z}.
\end{align}
Moreover, $H$ is unique up to postcomposition with M\"obius transformations.
\end{theorem}

The David Integrability Theorem is a generalization of the Measurable Riemann Mapping Theorem \cite{AstalaIwaniecMartin:quasiconformal}*{Theorem 5.3.4, p.~170}, which states that if $\|\mu\|_\infty<1$, then there exists a quasiconformal homeomorphism  $H\colon \widehat{\C} \to \widehat{\C}$ that solves the Beltrami equation \eqref{beltrami_equation}.

\begin{theorem}[Uniqueness; \cite{AstalaIwaniecMartin:quasiconformal}*{Theorem~20.4.19}]\label{theorem:stoilow}
Let $\Omega\subset \widehat{\C}$ be an open set and $f,g\colon \Omega\to \widehat{\C}$ be David maps with
\begin{align*}
\mu_f=\mu_g
\end{align*}
almost everywhere. Then $f\circ g^{-1}$ is a conformal map on $g(\Omega)$.  
\end{theorem}

\begin{proposition}[\cite{LyubichMerenkovMukherjeeNtalampekos:David}*{Proposition 2.5}]\label{prop:david_qc_invariance}
Let $f\colon U\to V$ be a David homeomorphism between open sets $U,V\subset \widehat {\C}$ and $g\colon V\to \widehat \C$ be a quasiconformal embedding. Then $g\circ f$ is a David map.
\end{proposition}

\begin{theorem}[\cite{LyubichMerenkovMukherjeeNtalampekos:David}*{Theorem 2.8}]\label{theorem:removable}
    Let $H\colon \widehat \C\to \widehat \C$ be a David homeomorphism. Then the set $E=H(\mathbb S^1)$ is locally conformally removable in the following sense. For every open set $U\subset \widehat \C$ and each topological embedding $f\colon U\to \widehat \C$ that is conformal in $U\setminus E$, we have that $f$ is conformal on $U$. 
\end{theorem}

\subsection{Extensions of circle homeomorphisms}
We will need an extension result for David maps, which is a generalization of the well-known extension of Beurling and Ahlfors \cite{BeurlingAhlfors:extension}. Let $h\colon \mathbb S^1\to \mathbb S^1$ be an orientation-preserving homeomorphism. We define the \textit{symmetric distortion function} of $h$ to be 
\begin{align*}
\rho_h(z,t)= \max\left\{  \frac{|h(e^{2\pi i t}z)-h(z)| }{ |h(e^{-2\pi i t}z)-h(z)| } ,  \frac{|h(e^{-2\pi i t}z)-h(z)| }{ |h(e^{2\pi i t}z)-h(z)| }\right\},
\end{align*}
where $z\in \mathbb S^1$ and $0<t<1/2$. The distortion function $\rho_h(z,t)$ measures whether $h$ maps adjacent arcs of equal length $t$ to adjacent arcs of equal length and how far it is from that behavior. We define the \textit{scalewise distortion function} of $h$ to be  
\begin{align*}
\varrho_h(t)= \max_{z\in \mathbb{S}^1}\rho_h(z,t),  
\end{align*}
where $0<t<1/2$. If $\varrho_h(t)$ is bounded above, then the function $h$ is a quasisymmetry and has a quasiconformal extension on $\D$, by the theorem of Beurling and Ahlfors \cite{BeurlingAhlfors:extension}. Zakeri observed in \cite{Zakeri:boundary}*{Theorem 3.1}, by applying a result of \cite{ChenChenHe:boundary}, that there is a growth condition on $\varrho_h(t)$ that is sufficient for a homeomorphism $h$ of the circle to have a David extension in the disk.

\begin{theorem}[David extension; \cite{ChenChenHe:boundary}*{Theorem 3}, \cite{Zakeri:boundary}*{Theorem 3.1}]\label{theorem:extension_david}
Let $h\colon \mathbb S^1\to \mathbb S^1$ be an orientation-preserving homeomorphism and suppose that
\begin{align*}
\varrho_h(t) = O(\log(1/t))\quad \textrm{as} \quad t\to 0.
\end{align*}
Then $h$ has an extension to a homeomorphism $\widetilde h\colon \bar \D\to \bar\D$ such that $\widetilde h|_{\D}$ is a David map. 
\end{theorem}

See also the recent work \cite{KarafylliaNtalampekos:extension} for a stronger extension result, which asserts that the weaker condition  
$$\exp \rho_h(\cdot,\cdot) \in L^p( \mathbb S^1 \times (0,1/2))\quad  \textrm{for some $p>0$}$$
is sufficient for an extension.  

\section{Markov partitions and expansive circle maps}\label{section:markov_expansive}

If $a,b\in \mathbb{S}^1$, we denote by $\arc{[a,b]}$ and $\arc{(a,b)}$ the closed and open arcs, respectively, from $a$ to $b$ in the positive orientation. The arc $\arc{(b,a)}$, for example, is the complementary arc of $\arc{[a,b]}$. We also denote the arc $\arc{(a,b)}$ by $\inter{\arc{[a,b]}} $. We say that two non-overlapping arcs $I,J\subset \mathbb{S}^1$ are \textit{adjacent} if they share an endpoint.

\begin{definition}[Markov partition]\label{definition:markov_partition}
A \textit{Markov partition} associated to a covering map $f\colon \mathbb{S}^1\to \mathbb{S}^1$ is a covering of the unit circle by closed arcs $A_k=\arc{[a_k, a_{k+1}]}$, $k\in \{0,\dots, r\}$, $r\geq 1$, where $a_{r+1}=a_0$, that have disjoint interiors and satisfy the following conditions.
\begin{enumerate}[label={(\roman*)}]
\item\label{markov:i} The map $f_k=f|_{\inter{A_k}}$ is injective for $k\in \{0,\dots,r\}$.
\item\label{markov:ii} If $f(\inter{A_k})\cap \inter{A_j}\neq \emptyset$ for some $k,j\in \{0,\dots,r\}$, then $\inter{A_j}\subset f(\inter{A_k})$. 
\item\label{markov:iii} The set $\{a_0,\dots,a_r\}$ is invariant under $f$.
\end{enumerate}    
We denote the above Markov partition by $\mathcal P(f;\{a_0,\dots,a_r\})$. 
\end{definition}

Note that by definition, the points $a_0,\dots,a_r$ are ordered in the positive orientation if $r\geq 2$; if $r=1$, there is no natural order. Moreover, \ref{markov:ii} and \ref{markov:iii} are equivalent under condition \ref{markov:i}.

\subsection{Admissible words}
Let $f\colon \mathbb{S}^1\to \mathbb{S}^1$ be a covering map and consider a Markov partition $\mathcal P=\mathcal P(f;\{a_0,\dots,a_r\})$. We can associate a matrix $B=(b_{kj})_{k,j=0}^r$ to $\mathcal P$ so that $b_{kj}=1$ if $f_k(A_k)\supset A_j$ and $b_{kj}=0$ otherwise. In the case $b_{kj}=1$ we define $A_{kj}$ to be the closure of $f_k^{-1}(\inter A_j)$. If $w=(j_1,\dots,j_n)\in \{0,\dots,r\}^n$, $n\in \N$, $k\in \{0,\dots,r\}$, and once $A_w$ has been defined, we define $A_{kw}$ to be the closure of $f_k^{-1}(\inter A_w)$ whenever $b_{kj_1}=1$. A \textit{word} $w=(j_1,\dots,j_n)\in \{0,\dots,r\}^n$, $n\in \N$, is \textit{admissible (for the Markov partition $\mathcal P$)} if $b_{j_1j_2} =\dots=b_{j_{n-1}j_n}=1$. The empty word $w$ is considered to be admissible and we define $A_w=\mathbb S^1$. We also define $A_w=\emptyset$ if $w$ is not admissible. The \textit{length} of a word $w=(j_1,\dots,j_n)\in \{0,\dots,r\}^n$ is defined to be $|w|=n$ and the length of the empty word is $0$. It follows from properties \ref{markov:i} and \ref{markov:ii} that for each $n\in \N$ the arcs $A_w$, where $|w|=n$, have disjoint interiors and their union is equal to $\mathbb{S}^1$. 

Inductively, we have $A_{wj}\subset A_w$ for all admissible words $w$ and $j\in \{0,\dots,r\}$. If $A_{wj}$ is non-empty, we say that $A_{wj}$ is a \textit{child} of $A_w$ and $A_w$ is the \textit{parent} of $A_{wj}$. Thus, $A_w$ has at most $r+1$ children. Also, if $k\in \{0,\dots,r\}$, $w$ is a non-empty word, and $(k,w)$ is admissible, then $f$ maps $\inter A_{kw}$ homeomorphically onto $\inter A_w$. Thus, $f$ acts as a subshift on the space of admissible words. 

\subsection{Levels and complementary arcs}
Let $f\colon \mathbb S^1\to \mathbb S^1$ be a covering map and $\mathcal P(f; \{a_0,\dots,a_r\})$ be a Markov partition. We let $F_n$ be the preimages of $F_1= \{a_0,\dots,a_r\}$ under $n-1$ iterations of $f$ and $F_{0}=\emptyset$. Observe that $F_{n}\supset F_{n-1}$ for each $n\in \N$. Indeed, $F_1\supset f(F_1)$ by condition \ref{markov:iii} in Definition \ref{definition:markov_partition}, hence $F_2=f^{-1}(F_1)\supset f^{-1}(f(F_1)) \supset F_1$, so the conclusion follows by induction. We define the \textit{level} of a point $c\in \bigcup_{n\geq 1}F_n$ to be the unique $n\in \N$ such that $c\in F_n\setminus F_{n-1}$. 

For each $n\in \N$, the set $\mathbb S^1\setminus F_n$ consists of open arcs. For practical purposes we will use the terminology \textit{complementary arcs of $F_n$} to indicate the family of the \textit{closures} of the components of $\mathbb S^1 \setminus F_n.$ Hence, all complementary arcs of $F_n$ are closed arcs. Note that the complementary arcs of $F_n$ are the arcs $A_w$, where $w$ is an admissible word with $|w|=n$.

\subsection{Canonical splitting}\label{section:canonical_splitting}
We continue to assume that  $\mathcal P(f; \{a_0,\dots,a_r\})$ is a Mar\-kov partition associated to $f$. Let $l\in \N$ and $w=(j_1,\dots,j_l)$ be an admissible word. Let $m\in \{0,\dots,l-1\}$ be such that $m+1$ is the minimum of the levels of the endpoints of $A_w$. We set $v=(j_1,\dots,j_m)$ and $u=(j_{m+1},\dots,j_l)$. We call $(v,u)$ the \textit{canonical splitting} of the admissible word $w$ and we write $w=(v,u)$. Note $v$ can be the empty word if $m=0$. Also, since $A_w$ is a complementary arc of $F_l$, we have $m+1\leq l$, so $|u|\geq 1$ and $u$ is not the empty word. Finally, note that $A_u$ has a point of $F_1$ as an endpoint, since it is the image of $A_w$ under $f^{\circ m}$. We define the canonical splitting $(v,u)$ of the empty word to consist of empty words.

\begin{lemma}\label{lemma:canonical_split}
Let $f\colon \mathbb{S}^1\to \mathbb{S}^1$ be a covering map and consider a Markov partition $\mathcal P(f; \{a_0,\dots,a_r\})$. For each admissible word $w$ with canonical splitting $w= (v,u)$ we have $A_w\subset \inter A_v$.
\end{lemma}
\begin{proof}
    If $v$ is the empty word we have nothing to show, so suppose that $|v|\geq 1$. Since $w = (v,u)$, we have $A_w \subset A_v$.
    Since the endpoints of $A_v$ have level at most $m$ and both endpoints of $A_w$ have level at least $m+1$, we conclude that $A_w$ and $A_v$ have no common endpoint. Hence, $A_w\subset \inter A_v$.
\end{proof}

\subsection{Expansive maps}We give the definition of an expansive map of $\mathbb S^1$.

\begin{definition}[Expansive map]\label{definition:expansive}
Let $f\colon \mathbb{S}^1\to \mathbb{S}^1$ be a continuous map. 
\begin{enumerate}[label=\normalfont(\alph*)]
    \item For a point $a\in \mathbb S^1$ we say that $f$ is \textit{expansive at $a$} if there exists $\delta>0$ such that for every $b\in \mathbb S^1$ with $a\neq b$ we have $|f^{\circ n}(a)-f^{\circ n}(b)|>\delta$ for some $n\in \N\cup \{0\}$.
    \item We say that $f$ is \textit{expansive} if there exists $\delta>0$ such that for every $a,b\in \mathbb{S}^1$ with $a\neq b$ we have $|f^{\circ n}(a)-f^{\circ n}(b)|>\delta$ for some $n\in \N\cup \{0\}$.
\end{enumerate}

\end{definition}

An expansive map of $\mathbb S^1$ is necessarily a covering map of degree $d\geq 2$ \cite{HemmingsenReddy:expansive}*{Theorem 2}. We now list some important properties of expansive maps.
\begin{enumerate}[label={($E$\arabic*)}]
\item\label{expansive:diameters}
Let $f\colon \mathbb S^1\to \mathbb S^1$ be a continuous map and  $\mathcal P(f;\{a_0,\dots,a_r\})$ be a Markov partition. Then $f$ is expansive if and only if
\begin{align*}
\lim_{n\to\infty}\max\{\diam{A_w}: |w|=n \} =0.
\end{align*}
This can be proved easily using \cite{PrzytyckiUrbanski:book}*{Theorem 3.6.1, p.~143}. 

\item\label{expansive:conjugate}Let $f\colon \mathbb S^1\to \mathbb S^1$ be an expansive covering map of degree $d$. Then there exists an orientation-preserving homeomorphism $h\colon \mathbb{S}^1\to \mathbb{S}^1$ that conjugates $f$ to either the map $g(z)=\bar z^d$ or the map $g(z)=z^d$, depending on whether $f$ is orientation-reversing or orientation-preserving, respectively. Moreover, $h$ is unique up to rotation by a $(d+1)$-st root of unity if $g(z)=\bar z^d$ or a $(d-1)$-st root of unity if $g(z)=z^d$.
This property  is a consequence of \ref{expansive:diameters}. A more general statement for expansive self-maps of a compact manifold can be found in \cite{CovenReddy:expansive}*{Property $(2')$, p.~99}. 

\item\label{expansive:iterate_iff}A map $f\colon \mathbb S^1\to \mathbb S^1$ is expansive if and only if $f^{\circ p}$ is expansive for each $p\in \N$.  This is also a consequence of \ref{expansive:diameters}.

\item\label{expansive:fixed}
If $f\colon \mathbb S^1\to \mathbb S^1$ is expansive at $z_0$ and $f(z_0)=z_0$, then there exists a neighborhood of $z_0$ in which $f$ has no {other} fixed points.

\item\label{expansive:iterate}If $f\colon \mathbb S^1\to \mathbb S^1$ is expansive at $z_0$ and $p\in \N$, then $f^{\circ p}$ is also expansive at $z_0$. This is a consequence of the uniform continuity of $f$.
\end{enumerate}

A refined version of property \ref{expansive:conjugate} that we will need for our considerations is the following lemma. Its proof is straightforward, based on property \ref{expansive:diameters}, and is omitted.

\begin{lemma}\label{lemma:expansive_conjugate_fg}
Let $f,g\colon \mathbb S^1\to \mathbb S^1 $ be expansive covering maps with the same orientation and $\mathcal P(f;\{a_0,\dots,a_r\})$, $\mathcal P(g;\{b_0,\dots,b_r\})$  be  Markov partitions. Consider the map $h\colon \{a_0,\dots,a_r\} \to \{b_0,\dots,b_r\}$ defined by $h(a_k)=b_k$ for $k\in \{0,\dots,r\}$ and suppose that $h$ conjugates the map $f$ to $g$ on the set $\{a_0,\dots,a_r\}$, i.e., 
\begin{align*}
h(f(a_k))=g(b_k)
\end{align*}
for $k\in \{0,\dots,r\}$. Then $h$ has an extension to an orientation-preserving homeomorphism of $\mathbb S^1$ that conjugates $f$ to $g$ on $\mathbb S^1$.
\end{lemma}

\begin{lemma}\label{lemma:expansive_point}
    Let $f\colon \mathbb S^1\to \mathbb S^1$ be a covering map let $z_0\in \mathbb S^1$ be a fixed point of $f$. Then $f$ is expansive at $z_0\in \mathbb S^1$ if and only if the following statement is true. 

    \smallskip
    \noindent
    There exist closed arcs $I,J$ with  $z_0\in \inter I\subset I \subset \inter J$ such that $f|_I\colon I\to J$ is bijective, and if we set $g=(f|_I)^{-1}\colon J\to I$, then $\{g^{\circ n}\}_{n\in \N}$ converges uniformly in $J$ to the constant function $z\mapsto z_0$.
\end{lemma}

\begin{proof}Suppose that $f$ is expansive at $z_0$. This assumption and the fact that $f$ is locally injective imply that there exist closed arcs $I,J$ with $z_0\in \inter I\subset I\subset \inter J$ such that $f|_I \colon I\to J$ is bijective. By shrinking $I$, we may assume that it does not contain any fixed points of $f$ and $f\circ f$ other than $z_0$. Let $g=(f|_{I})^{-1}$ and define $J=\arc{[a_0,b_0]}$ and $I_n=g^{\circ n}(J)=\arc{[a_n,b_n]}$, $n\in \N$. Note that $I_{n+1}\subset I_n$, $n\in \N$. We claim that $\bigcap_{n=1}^\infty I_n=\{z_0\}$.  If not, there exists a non-degenerate arc $[a,b]\subset \bigcap_{n=1}^\infty I_n$ that contains $z_0$. Without loss of generality, suppose that $b\neq z_0$. Note that $b_n\to b$ as $n\to\infty$. If $f$ is orientation-preserving, we have $f(b_{n+1})= b_n$, $n\in \N$, so $f(b)=b$; otherwise, we have $(f\circ f)(b_{2n+2})=b_{2n}$, $n\in \N$, so $(f\circ f)(b)=b$.  This is a contradiction.     

For the converse, it suffices to show that for each point $w_0\in I\setminus \{z_0\}$ there exists $n\in \N$ such that $f^{\circ n}(w_0)\notin J$. Suppose that a point $w_0\in I$ has the property that $w_n=f^{\circ n}(w_0)\in J$ for each $n\in \N$. Then $w_0=g^{\circ n}(w_n)$. Since $\{g^{\circ n}\}_{n\in \N}$ converges uniformly to the constant function $z\mapsto z_0$, we conclude that $g^{\circ n}(w_n)$ converges to $z_0$ as $n\to\infty$. Therefore $w_0=z_0$.
\end{proof}

\subsection{Primitive Markov partitions}
Let $f\colon \mathbb S^1\to \mathbb S^1$ be a covering map with a Markov partition $\mathcal P=\mathcal P(f;\{a_0,\dots,a_r\})$. We say that $\mathcal P$ is \textit{primitive} if for each $k\in \{0,\dots,r\}$ there exists $n\in \N$ such that $f^{\circ n}(A_k)=\mathbb S^1$. {Equivalently, the matrix $B$ associated to the Markov partition is primitive, i.e., there exists $n\in \N$ so that each entry in $B^n$ is positive.}

Recall that if $f$ is expansive, then it is conjugate to $z\mapsto z^d$ or $z\mapsto \overline z^d$; see \ref{expansive:conjugate}. Hence, in that case every Markov partition is primitive. On the other hand, not every covering map with a primitive Markov partition is expansive. 

\begin{lemma}\label{lemma:primitive}
    Let $f\colon \mathbb S^1\to \mathbb S^1$ be a covering map and $\mathcal P(f;\{a_0,\dots,a_r\})$ be a Markov partition. The Markov partition is primitive if and only if there exists $p\in \N$ such that each complementary arc of $F_n$, where $n\in \N$, contains at least two complementary arcs of $F_{n+p}$. In that case, one may take $p=r$. 
\end{lemma}

\begin{proof}
    Suppose that the Markov partition is primitive. It suffices to prove that each complementary arc $A$ of $F_1$ contains at least two complementary arcs of $F_{r+1}$. Let $A$ be a complementary arc of $F_1$. We show that one of the arcs $f(A),\dots,f^{\circ r}(A)$ contains at least two complementary arcs of $F_1$. If not, then each of $A,f(A),\dots,f^{\circ r}(A)$ is a complementary arc of $F_1$. Consider two cases. 

    \medskip
    \noindent
    \textbf{Case 1:}  Two of the arcs $A,f(A),\dots,f^{\circ r}(A), f^{\circ (r+1)}(A)$ are identical. Then there exists a complementary arc $B$ of $F_1$ and $p\in \N$ such that $f^{\circ p}$ maps $B$ onto itself and is injective in the interior of $B$. In particular, $f^{\circ (np)}$ maps $B$ onto itself for each $n\in \N$. This contradicts the assumption that the Markov partition is primitive.
    
    \medskip
    \noindent
    \textbf{Case 2:} The arcs $A,f(A),\dots,f^{\circ r}(A), f^{\circ (r+1)}(A)$ are distinct. Then, the complementary arcs of $F_1$ are precisely the arcs $A,f(A),\dots,f^{\circ r}(A)$. By the above and the properties of a Markov partition, $f$ maps injectively $\mathbb S^1\setminus f^{\circ r}(A)$ onto $\mathbb S^1\setminus A$ and $f$ is injective in the interior of $f^{\circ r}(A)$. Thus $f(f^{\circ r}(A))\supset A$. By the assumption of Case 2, we have $f^{\circ (r+1)}(A)\neq A$, so $f(f^{\circ r}(A))\supsetneq A$. Therefore, each point of $\inter A$ has a unique preimage under $f$, but there exist points outside $A$ with at least two preimages, one in $\mathbb S^1\setminus f^{\circ r}(A)$ and one in $f^{\circ r}(A)$. This is a contradiction to the fact that $f$ is a covering map.

    \medskip
    
    Conversely, suppose that there exists $p\in \N$ such that each complementary arc of $F_n$, where $n\in \N$, contains at least two complementary arcs of $F_{n+p}$. Let $A$ be a complementary arc of $F_1$. For each $k\in \N$, $A$ contains at least $2^k$ complementary arcs of $F_{kp+1}$. Each of those arcs is mapped by $f^{\circ (kp)}$ to a complementary arc of $F_1$. Hence, if $2^k\geq r+1$, then, given that $f$ is a covering map, we see that $f^{\circ (kp)}(A)$ covers at least $r+1$ complementary arcs of $F_1$, so it is equal to $\mathbb S^1$.  
\end{proof}

\subsection{Hyperbolic and parabolic points}\label{section:hyperbolic_parabolic_points}

We say that a finite collection of non-overlapping arcs $I_1,\dots,I_m\subset \mathbb{S}^1$, $m\in \N$, consists of \textit{consecutive} arcs if every arc $I_i$ shares one endpoint with another arc $I_j$, $j\neq i$. After renumbering if necessary, we will use the convention that $I_i$ has a common endpoint with $I_{i+1}$ for $i=1,\dots,m-1$. 

Let $f\colon \mathbb S^1\to \mathbb S^1$ be a covering map and $\mathcal P(f;\{a_0,\dots,a_r\})$ be a Markov partition.

\begin{definition}[Hyperbolic points]\label{definition:hyperbolic}
Let $a \in\{a_0,\dots,a_r\}$. We say that $a^+$ (resp.\ $a^-$) is \textit{hyperbolic} if there exist $\lambda>1$ and $L\geq 1$ such that the following statement is true. 

\smallskip
\noindent
If $I_1,I_2$ are consecutive complementary arcs of $F_n$, $n\geq 1$, $a$ is an endpoint of $I_1$, and $I_1,I_2\subset \arc{[a,z_0]}$ (resp.\ $I_1,I_2\subset \arc{[z_0,a]}$) for some $z_0\neq a$, then for each $i\in \{1,2\}$ we have
\begin{align}\label{inequality:adjacenthyperbolic}
L^{-1}\lambda^{-n}&\leq \diam {I_i} \leq L\lambda^{-n}.
\end{align}
In that case we set $\lambda(a^+)=\lambda$ (resp.\ $\lambda(a^-)=\lambda$) and call this number the \textit{multiplier} of $a^+$ (resp.\ $a^-$).
\end{definition}

\begin{definition}[Parabolic points]\label{definition:parabolic}
Let $a \in\{a_0,\dots,a_r\}$. We say that $a^+$  (resp.\ $a^-$) is \textit{parabolic} if there exist $N\in \N$ and $L\geq 1$ such that the following statement is true.

\smallskip
\noindent
If $I_1,I_2$ are consecutive complementary arcs of $F_n$, $n\geq 1$, $a$ is an endpoint of $I_1$, and $I_1,I_2\subset \arc{[a,z_0]}$ (resp.\ $I_1,I_2\subset \arc{[z_0,a]}$) for some $z_0\neq a$, then we have
\begin{align}\label{inequality:parabolic_1}
\frac{L^{-1}}{n^{1/N}}&\leq \diam {I_1} \leq \frac{L}{n^{1/N}} \quad \textrm{and}\quad \frac{L^{-1}}{n^{1/N +1}}\leq \diam {I_2} \leq \frac{L}{n^{1/N +1}}.
\end{align} 
In that case we set $N(a^+)=N$ (resp.\ $N(a^-)=N$) and call the number $N(a^+)+1$ (resp.\ $N(a^-)+1$) the \textit{multiplicity} of $a^+$ (resp.\ $a^-$).
\end{definition}

\begin{definition}\label{definition:parabolic_hyperbolic_symmetric}
    Let $a\in \{a_0,\dots,a_r\}$. We say that $a$ is \textit{symmetrically hyperbolic} if $a^+$ and $a^-$ are hyperbolic with $\lambda(a^+)=\lambda(a^-)$. In this case we denote by $\lambda(a)$ the common multiplier. We say that $a$ is \textit{symmetrically parabolic} if $a^{+}$ and $a^-$ are parabolic with $N(a^+)=N(a^-)$. In this case we denote this common number by $N(a)$. 
\end{definition}

\begin{remark}\label{remark:hp_expansive}
    Let $a\in \{a_0,\dots,a_r\}$ be a periodic point. If each of $a^+$ and $a^-$ is either hyperbolic or parabolic, then by Lemma \ref{lemma:expansive_point}, $f$ is expansive at $a$.
\end{remark}

\begin{remark}\label{remark:hyp_par_analytic}
If $f\colon \mathbb S^1\to \mathbb S^1$ is analytic and expansive, it is shown in Lemma 4.17 and Lemma 4.18 in \cite{LyubichMerenkovMukherjeeNtalampekos:David} that each point $a\in \{a_0,\dots,a_r\}$ is either symmetrically hyperbolic or symmetrically parabolic. In that case, if $a$ is periodic and $q$ is its period, then $\lambda(a)^q$ (resp.\ $N(a)+1$) is the usual multiplier (resp.\ multiplicity) of the analytic map $f^{\circ q}$.
\end{remark}

\section{Main extension theorem}\label{section:extension}
Our main extension theorem provides conditions so that a homeomorphism $h\colon \mathbb S^1\to \mathbb S^1$ that conjugates two covering maps of $\mathbb S^1$ extends to a quasiconformal or David homeomorphism of $\mathbb D$. 

Let $f,g\colon \mathbb S^1\to \mathbb S^1$ be covering maps with corresponding Markov partitions $\mathcal P(f;\{a_0,\dots,a_r\})$, $\mathcal P(g;\{b_0,\dots,b_r\})$. We give the definitions below using the notation associated to the map $f$. First we require the following condition.
\begin{enumerate}[label=\normalfont(M1)]
\medskip
    \item\label{condition:hp} For each $a\in \{a_0,\dots,a_r\}$, each of $a^+$ and $a^-$ is either hyperbolic or parabolic, as defined in Section \ref{section:hyperbolic_parabolic_points}.
\medskip
\end{enumerate}
Define $A_k=\arc{[a_k,a_{k+1}]}$ for $k\in \{0,\dots,r\}$ and recall that $f_k=f|_{\inter{A_k}}$ is injective by the definition of a Markov partition. We impose the following condition.
\begin{enumerate}[label=\normalfont(M2)]
\medskip 
    \item\label{condition:uv} For each $k\in \{0,\dots,r\}$ there exist open neighborhoods $U_k$ of $\inter{A_k}$ and $V_k$ of $f_k(\inter{A_k})$ in $\C$ such that $f_k$ has an extension to a {homeomorphism} from $U_k$ onto $V_k$. We denote the extension by $f_k$. Furthermore, we assume that
    \begin{align*}
    \bigcup_{\substack{0\leq j\leq r\\(k,j) \,\, \textrm{admissible}}}U_j \subset V_k
    \end{align*} 
    for all $k\in \{0,\dots,r\}$ and that the regions $U_k$, $k\in \{0,\dots,r\}$, are pairwise disjoint. 
\medskip
\end{enumerate}
We denote by $f$ the map that is equal to $f_k$ on $U_k$, $k\in \{0,\dots,r\}$. 
Using condition \ref{condition:uv}, we can define $U_{kj}=f_k^{-1}(U_j)$ for $k,j\in\{0,\dots,r\}$, whenever $(k,j)$ is admissible; recall the definitions given after Definition \ref{definition:markov_partition}. Note that $U_{kj}\subset U_k$ and $f_k$ maps $U_{kj}$ onto $U_j$. Inductively, for each admissible word $w$ we can find open regions $U_w$ with the following properties:
\begin{enumerate}[label=\normalfont{(\roman*)}]
\item $U_{wj} \subset U_w$, if $(w,j)$ is admissible, and 
\item $f_k$ maps $U_{kw}$ homeomorphically onto $U_w$, if $(k,w)$ is admissible. 
\end{enumerate}
If $w=(k_1,\dots, k_n)$ is admissible, we define $f_w=f_{k_n}\circ\dots \circ f_{k_1}=f^{\circ n}$ on the set $\bigcup \{U_{wj}: 0\leq j\leq r,\, (w,j)\,\, \textrm{admissible}\}$. It follows that $f_w$ maps homeomorphically $U_{wj}$ onto $U_j$. Observe that for each admissible word $w$ the open arc $\inter{A_w}$ is contained in $U_w$. Next, we impose the following condition.
\begin{enumerate}[label=\normalfont(M3)]
\medskip
    \item\label{condition:qs} There exists $K\geq 1$ such that for each $m\in \N\cup \{0\}$ and each admissible word $w$ with $|w|\geq m$ the map $f^{\circ m}|_{U_w}$ is $K$-quasiconformal.  
\medskip 
\end{enumerate}
Finally, we consider a technical condition that is a {refinement} of \ref{condition:qs} and compensates for the unavailability of Koebe's distortion theorem in our setting. Recall the definition of the canonical splitting of a word from Section \ref{section:canonical_splitting}.
\begin{enumerate}[label=\normalfont(M$3^\ast$)]
\medskip
    \item  \label{condition:qs_strong} There exists $C>0$ such that for each admissible word $w$ with canonical splitting $w=(v,u)$, if $|v|=m$ and $|u|=n$, then the map $f^{\circ m}|_{U_w}$ is $K_n$-quasiconformal for $K_n=1+C(1+\log (n+1))^{-1}$.
\medskip
\end{enumerate}

We now state the main extension theorem.

\begin{theorem}[Main extension theorem]\label{theorem:extension_generalization}Let $f,g\colon \mathbb{S}^1\to \mathbb{S}^1$ be expansive covering maps with the same orientation and $\mathcal P(f;\{a_0,\dots,a_r\})$, $\mathcal P(g;\{b_0,\dots,b_r\})$ be Markov partitions satisfying conditions \ref{condition:hp}, \ref{condition:uv}, and \ref{condition:qs}. Suppose that the map $h\colon \{a_0,\dots,a_r\} \to \{b_0,\dots,b_r\}$ defined by $h(a_k)=b_k$, $k\in \{0,\dots,r\}$, conjugates $f$ to $g$ on the set $\{a_0,\dots,a_r\}$ and assume that for each point $a\in \{a_0,\dots,a_r\}$ and for $b=h(a)$ one of the following alternatives occurs.

\smallskip

\begin{enumerate}[label= {\textup{\scriptsize(\textbf{H/P${\rightarrow}$H/P})}}, leftmargin=6em] 
\item\label{HH} There exists $\mu=\mu(a)>0$ such that if $a^{\pm}$ is parabolic then $b^{\pm}$ is parabolic with 
$\mu^{-1} N(a^{\pm})=N(b^{\pm})$, and if $a^{\pm}$ is hyperbolic then  $b^{\pm}$ is hyperbolic with  $\lambda(a^{\pm})^{\mu}=\lambda(b^{\pm})$.
\end{enumerate}

\begin{enumerate}[label={\textup{\scriptsize(\textbf{H$\to$P})}}, leftmargin=6em]
\item\label{HP} $a$ is symmetrically hyperbolic and $b$ is symmetrically parabolic.
\end{enumerate}

\smallskip 
\noindent
If the alternative \ref{HP} does not occur, then the map $h$ extends to a homeomorphism $\widetilde h$ of $\bar \D$ such that $\widetilde h|_{\mathbb S^1}$ conjugates $f$ to $g$ and $\widetilde h|_{\D}$ is quasiconformal. 

\smallskip 
\noindent
If $g$ satisfies condition \ref{condition:qs_strong}, then the map $h$ extends to a homeomorphism $\widetilde h$ of $\bar \D$ such that $\widetilde h|_{\mathbb S^1}$ conjugates $f$ to $g$ and $\widetilde h|_{\D}$ is a David map. 
\end{theorem}

\begin{remark}We note that alternative \ref{HH} covers the following cases:
\begin{itemize}
\item $a^-,b^-$ are hyperbolic and $a^+,b^+$ are hyperbolic with $\lambda(a^{\pm})^{\mu}=\lambda(b^{\pm})$
\item $a^-,b^-$ are hyperbolic and $a^+,b^+$ are parabolic with $\lambda(a^-)^{\mu}=\lambda(b^-)$ and $\mu^{-1}N(a^+)=N(b^+)$
\item $a^-,b^-$ are parabolic and $a^+,b^+$ are hyperbolic with $\mu^{-1}N(a^-)=N(b^-)$ and $\lambda(a^+)^{\mu}=\lambda(b^+)$
\item $a^-,b^-$ are parabolic and $a^+,b^+$ are parabolic with $\mu^{-1}N(a^{\pm})=N(b^{\pm})$
\end{itemize} 
The only restriction is that the multiplicities and multipliers of the points $a^+,a^-$ and $b^+,b^-$ have to be related by the same number $\mu$, which depends only on the point $a$.
\end{remark}

\begin{remark}
    We can relax the assumption that $f$ and $g$ are expansive to the weaker conditions that $f$ and $g$ are covering maps and the the associated Markov partitions are primitive. Condition \ref{condition:hp} implies that $f,g$ are expansive at each periodic point of the Markov partition; see Remark \ref{remark:hp_expansive}. Upon imposing conditions \ref{condition:uv} and \ref{condition:qs}, we see that $f$ and $g$ satisfy the assumptions of  Theorem \ref{theorem:expansive_criterion} below. Therefore they are expansive.
\end{remark}

The theorem generalizes the recent result of \cite{LyubichMerenkovMukherjeeNtalampekos:David}, which instead assumes that $f$ and $g$ are piecewise analytic. One of the main technicalities here is that we do not have Koebe's distortion theorem available, since the maps $f$ and $g$ are piecewise quasiconformal. The other technical difficulty is that our definition of hyperbolic and parabolic points (Section \ref{section:hyperbolic_parabolic_points}) is much weaker than the corresponding definitions in the case of analytic maps.

We initiate the proof of Theorem \ref{theorem:extension_generalization}, which occupies the rest of the section. We will formulate most of the statements using the notation associated to the map $f$ and the Markov partition $\mathcal P(f;\{a_0,\dots,a_r\})$. The maps $f,g$ are assumed, throughout, to be covering maps of $\mathbb S^1$ with the same orientation. At the moment, we do not assume that they are expansive, unless otherwise stated. We will formulate a series of lemmas, indicating each time which of the conditions \ref{condition:hp}, \ref{condition:uv}, \ref{condition:qs}, and \ref{condition:qs_strong} are assumed.

\subsection{Quasiconformal elevator}

Let $I\subset \mathbb S^1$ be a closed non-degenerate arc and suppose that there exists an admissible word $w$ with the property that $I\subset A_w$ and $I$ is not contained in any child of $A_w$. We say that $w$ \textit{is the word associated to} $I$. If $f$ is expansive, then each non-degenerate closed arc has a unique associated word by \ref{expansive:diameters}. Note that $w$ could be the empty word, in which case $A_w=\mathbb S^1$. We start with a combinatorial lemma.

\begin{lemma}\label{lemma:split}
Suppose that $\mathcal P(f;\{a_0,\dots,a_r\})$ is a primitive Markov partition. Let $I\subset \mathbb S^1$ be a non-degenerate closed arc. Let $w$ be the word associated to $I$ and set $l=|w|\in \N\cup \{0\}$. 
\begin{enumerate}[label=\normalfont (\Roman*)]
    \item\label{lemma:split:ai} If $I$ contains a complementary arc of $F_{l+r+1}$, then it also contains a complementary arc of $F_{l+r+1}$ that is disjoint from the endpoints of $A_w$.
    \item\label{lemma:split:aii} If $I$ does not contain any complementary arc of $F_{l+r+1}$, then the following statements are true.
\begin{enumerate}[label=\normalfont (\arabic*)]
    \item\label{lemma:split:i} There exists a unique point $c\in I\cap F_{l+1}$. Moreover, $c\in \inter I$ and $c$ splits $I$ into two closed arcs $I^-$ and $I^+$.
    \item\label{lemma:split:ii} There exist two complementary arcs of $F_{l+r+1}$ containing the endpoints of $A_w$ that are disjoint from $I$.
    \item\label{lemma:split:iii} If $w^{\pm}$ is the word associated to $I^{\pm}$, then the minimum of the levels of the endpoints of $A_{w^{\pm}}$ is $l+1$. In particular, the canonical splitting of $w^{\pm}$ has the form $w^{\pm}=(w,u^{\pm})$.
    \item\label{lemma:split:iv} $I^{\pm}$ contains a complementary arc of $F_{|w^{\pm}|+1}$ that has $c$ as an endpoint. 
\end{enumerate}
\end{enumerate}
\end{lemma}

\begin{proof}
    Recall that $A_w$ is a complementary arc of $F_l$ and $A_w\supset I$. The children of $A_w$ are complementary arcs of $F_{l+1}$. Since $I$ is not contained in any child of $A_w$, there exists a point $c\in \inter I \cap F_{l+1}$. In particular, $A_w$ contains two complementary arcs $A^{+},A^{-}$ of $F_{l+1}$ having $c$ as an endpoint.  By Lemma \ref{lemma:primitive}, $A^{+}$ (resp.\ $A^-$) contains two complementary arcs $B^{+}_{i}$ (resp.\ $B^-_i$), $i=1,2$, of $F_{l+r+1}$, such that $c$ is an endpoint of $B_1^+$ (resp.\ $B^-_1$). See the bottom part of Figure \ref{fig:split}.  

    Suppose that $I$ contains a complementary arc of $F_{l+r+1}$, as in \ref{lemma:split:ai}. Then it must contain either $B_{1}^-$ or $B_1^+$. None of these arcs contains the endpoints of $A_w$, so the proof of this case is completed.

    Next, suppose that $I$ does not contain any complementary arc of $F_{l+r+1}$ as in \ref{lemma:split:aii}. Thus, the same holds for complementary arcs of $F_{l+1}$. As a consequence, there exists a unique point in $I\cap F_{l+1}$, namely the point $c\in \inter I\cap F_{l+1}$ chosen above. Denote by $I^{\pm}$ the closures of the components of $I\setminus \{c\}$, so that $I^{\pm}\cap A^{\pm}$ has non-empty interior. We have completed the proof of \ref{lemma:split:i}.

    By assumption, none of the complementary arcs $B_{i}^{\pm}$, $i=1,2$, of $F_{l+r+1}$ is contained in $I^{\pm}$. Hence, $B_2^{\pm}$ is disjoint from $I$.  This completes the proof of \ref{lemma:split:ii}.

    Since $B^{\pm}_1$ is not a complementary arc of $F_{l+1}$, the endpoint of $B^{\pm}_1$ that is different from $c$ has level larger than $l+1$. Therefore the minimum of the levels of the endpoints of $B^{\pm}_1$ us $l+1$. Now, if $w^{\pm}$ is the word associated to $I^{\pm}$, we have $A_{w^{\pm}}\subset B^{\pm}_1$. Hence, the minimum of the levels of the endpoints of $A_{w^{\pm}}$ is also $l+1$. Consider the canonical splitting $w^{\pm}=(v^{\pm}, u^{\pm})$, where $|v^{\pm}|=l=|w|$. Since $A_{w^{\pm}}\subset A_w$, the first $l$ letters of $w^{\pm}$ coincide with those of $w$. We conclude that $v^{\pm}=w$. This proves \ref{lemma:split:iii}.

    Finally, since $I^{+}\subset A_{w^{+}}$ and $c$ is an endpoint of $A_{w^{+}}$, the definition of $w^{+}$ implies that the child of $A_{w^+}$ that has $c$ as an endpoint is contained in $I^+$. The analogue of this statement is true for $I^-$, as required in \ref{lemma:split:iv}.    
\end{proof}

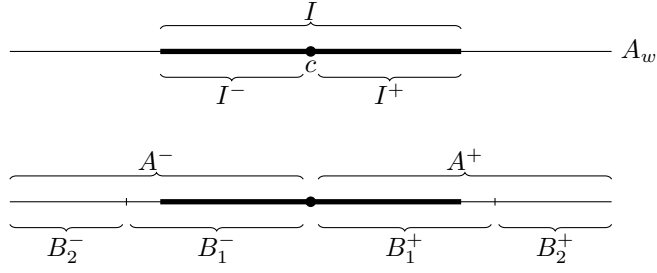
\begin{figure}
    \centering
    \begin{tikzpicture}
        \begin{scope}
            \draw (0,0)--(8,0) node[anchor=west] {$A_w$};
            \draw[line width=2pt] (2,0.0)--(6,0.0);
            \fill (4,0) circle (2pt) node[anchor=north] {$c$};

            \draw [decorate, decoration = {brace}] (2,0.3) --  (6,0.3) node[pos=0.5, anchor=south]{$I$};
            \draw [decorate, decoration = {brace,mirror}] (2,-0.3) --  (3.9,-0.3) node[pos=0.5, anchor=north]{$I^-$};
            \draw [decorate, decoration = {brace,mirror}] (4.1,-0.3) --  (6,-0.3) node[pos=0.5, anchor=north]{$I^+$};
        \end{scope}

        \begin{scope}[shift={(0,-2)}]
            \draw (0,0)--(8,0);
            \draw[line width=2pt] (2,0.0)--(6,0.0);
            \fill (4,0) circle (2pt);

            \draw [decorate, decoration = {brace}] (0,0.3) --  (3.9,0.3) node[pos=0.5, anchor=south]{$A^-$};
            \draw [decorate, decoration = {brace}] (4.1,0.3) --  (8,0.3) node[pos=0.5, anchor=south]{$A^+$};

            \draw [decorate, decoration = {brace,mirror}] (0,-0.3) --  (1.5,-0.3) node[pos=0.5, anchor=north]{$B_2^-$};
            \draw [decorate, decoration = {brace,mirror}] (1.6,-0.3) --  (3.9,-0.3) node[pos=0.5, anchor=north]{$B_1^-$};
            \draw [decorate, decoration = {brace,mirror}] (4.1,-0.3) --  (6.4,-0.3) node[pos=0.5, anchor=north]{$B_1^+$};
            \draw [decorate, decoration = {brace,mirror}] (6.5,-0.3) --  (8,-0.3) node[pos=0.5, anchor=north]{$B_2^+$};

            \draw (1.55,0.05)--(1.55,-0.05); 
            \draw (6.45,0.05)--(6.45,-0.05);
         \end{scope}    
    \end{tikzpicture}
    \caption{Illustration of the proof of Lemma \ref{lemma:split}.}
    \label{fig:split}
\end{figure}

For the proof of Theorem \ref{theorem:extension_generalization} we have to study the behavior of the conjugating homeomorphism $h$ in small arcs $I\subset \mathbb S^1$. Starting with an arbitrary non-degenerate arc $I\subset \mathbb S^1$, we map it quasisymmetrically, by applying a suitable iterate of $f$, to an arc $I'$ that is located near a point $a\in F_1$. This procedure is referred to as the \textit{quasiconformal elevator} and is described more precisely in Lemma \ref{lemma:newelevator}.  Note that there is a fundamental dichotomy: either $I'$ contains the point $a$ (as in  \ref{lemma:newelevator:2} below), or $I'$ lies only on one side of $a$ (as in \ref{lemma:newelevator:1} below).  We remark that the presence of parabolic points does not allow us to blow up the arc $I$ quasisymmetrically to an arc of large diameter, comparable to $1$.

\begin{lemma}[Quasiconformal elevator]\label{lemma:newelevator}
    Suppose that $\mathcal P(f;\{a_0,\dots,a_r\})$ is a primitive Markov partition that satisfies conditions \ref{condition:uv} and \ref{condition:qs}. There exist $M,\gamma\geq 1$ and a distortion function $\eta$ such that for any non-degenerate closed arc $I\subset \mathbb S^1$ one of the following alternatives holds. Let $w$ be the word associated to $I$.
    
    \begin{enumerate}[label=\normalfont ($A$-\roman*)]
    \item\label{lemma:newelevator:1} The arc $I$ contains a complementary arc of $F_{|w|+r+1}$. Then there exist consecutive complementary arcs $I_1,\dots,I_p$ of $F_{|w|+r+1}$, where $p\leq M$, with
    \begin{align*}
        I_{i_0} \subset I \subset A_w= \bigcup_{i=1}^{p} I_i \quad \textrm{for some $i_0\in \{2,\dots,p-1\}$. }
    \end{align*}
    Let $w=(v,u)$ be the canonical splitting of $w$, and set $m=|v|$ and $n=|u|$. Then $A_u$ has a point $a\in F_1$ as an endpoint and 
    \medskip
    \begin{enumerate}[label=\normalfont (\alph*)]
        \item $f^{\circ m}|_{A_w}\colon A_w\to A_u$ is  $\eta$-quasisymmetric or
        \item $f^{\circ m}|_{A_w}\colon A_w\to A_u$ is  $\eta_n$-quasisymmetric under condition \ref{condition:qs_strong}, where $\eta_n(t)=M \max\{t^{K_n}, t^{1/K_n} \}$ for $K_n=1+M(1+\log (n+1))^{-1}$.
    \end{enumerate}
   
\medskip

    \item\label{lemma:newelevator:2} The arc $I$ does not contain any complementary arc of $F_{|w|+r+1}$. Then there exist non-overlapping arcs $I^-,I^+$ such that $I=I^-\cup I^+$ and if $w^{\pm}$ is the word associated to $I^{\pm}$, then the canonical splitting is of the form $w^{\pm}=(w, u^{\pm})$. Moreover, for $m=|w|$, $f^{\circ m}$ maps the common endpoint of $I^-$ and $I^+$ to a point $a\in F_1$ and $I^{\pm}$ contains a complementary arc of $F_{|w^{\pm}|+r+1}$, so the first alternative \ref{lemma:newelevator:1} is applicable to each of $I^{\pm}$. Finally, $f^{\circ m}|_I$ is $\eta$-quasisymmetric.
    \end{enumerate}
In both cases, we have 
\begin{align}\label{lemma:newelevator:diam}  \displaystyle{\diam{f^{\circ m}(I)}\geq M^{-1} (\diam{I})^{\gamma}}.
\end{align}
\end{lemma}

\begin{proof}
    For each $j\in \{0,\dots,r\}$, let $B_j$ be the union of all complementary arcs of $F_{r+1}$ that are (compactly) contained in the open arc $f(\inter A_j)\subset V_j$. Let $B_j',V_j'$ be topological disks such that $B_j\subset B_j'\subset \overline{B_j'}\subset V_j'\subset V_j$. We set
    $$ \delta =\min\{ \mod(V_j'\setminus \overline{B_j'}) : j=0,\dots,r \}.$$

    Suppose that $I$ satisfies alternative \ref{lemma:newelevator:1}. The arc $A_w$ is the union of at most $(r+1)^{r+1}$ consecutive complementary arcs of $F_{|w|+r+1}$. By Lemma \ref{lemma:split} \ref{lemma:split:ai}, one of these arcs that is disjoint from the endpoints of $A_w$ is contained in $I$. This completes the proof of the first part of \ref{lemma:newelevator:1} regarding the inclusions. 
    
    Consider the canonical splitting $w=(v,u)$, and set $m=|v|$ and $n=|u|$. Observe that the case $m=0$ is trivial, since then $f^{\circ m}$ is the identity map. Suppose that $m\geq 1$, in which case we must have $n\geq 1$; see Section \ref{section:canonical_splitting}.
    
    By \ref{condition:uv}, $f^{\circ m}$ is a homeomorphism from $U_v$ onto a region $V_j$, $j\in \{0,\dots,r\}$. By Lemma \ref{lemma:canonical_split}, we have $A_w\subset \inter A_v\subset U_v$. Thus,  $A_w$ is mapped under $f^{\circ m}$ onto $A_u\subset f(\inter A_j)\subset  V_j$. In addition, since $m+1$ is the minimum of the levels of the endpoints of $A_w$, the arc $A_u$ has a point of $F_1$ as an endpoint. Since $B_j$ contains the union of complementary arcs of $F_{r+1}$ that are contained in the open arc $f(\inter A_j)\subset V_j$, we conclude that $A_u\subset B_j$. Under condition \ref{condition:qs}, $f^{\circ m}|_{U_v}$ is $K$-quasiconformal for some uniform $K\geq 1$. Denote by $\omega$ the inverse of $f^{\circ m}|_{U_v}$, restricted to $B_j$. By Theorem \ref{theorem:qcqs}, $\omega$ and its inverse are $\eta$-quasisymmetric, where 
    \begin{align}\label{lemma:newelevator:eta}
        \eta(t)= C(K,\delta) \max\{t^K,t^{1/K}\}.
    \end{align}
    In particular, $f^{\circ m}|_{A_w}$ is $\eta$-quasisymmetric. If condition \ref{condition:qs_strong} is assumed instead, the same argument as above gives the desired conclusion with $\eta_n$ in place of $\eta$. 
    
    Next, assume that $I$ satisfies the second alternative \ref{lemma:newelevator:2}. By Lemma \ref{lemma:split} \ref{lemma:split:i}, there exist non-degenerate and non-overlapping arcs $I^-,I^+$ such that $I=I^-\cup I^+$, so that the common endpoint $c$ of $I^-$ and $I^+$ lies in $F_{m+1}$, where $m=|w|$.  By Lemma \ref{lemma:split} \ref{lemma:split:iii}, the canonical splitting of $w^{\pm}$ has the form $w^{\pm}=(w, u^{\pm})$. Note that $f^{\circ m}$ maps $c$ to a point of $F_1$. By Lemma \ref{lemma:split} \ref{lemma:split:iv}, $I^{\pm}$ contains a complementary arc of $F_{|w^{\pm}|+r+1}$, so the first alternative is applicable to each of $I^{\pm}$. 
        
    Next, we justify that $f^{\circ m}|_I$ is quasisymmetric. By \ref{condition:uv}, $f^{\circ m}$ is a homeomorphism from $U_w$ onto a region $V_j$, $j\in \{0,\dots,r\}$.  By Lemma \ref{lemma:split} \ref{lemma:split:ii}, $I$ is contained in a compact subset of $A_w$ that is separated from the endpoints of $A_w$ by two complementary arcs of $F_{|w|+r+1}$. Thus, $f^{\circ m}$ maps $I$ homeomorphically onto an arc $I'$ that is separated from the endpoints of the arc $f^{\circ m} (\inter A_w)=f(\inter A_j)$ by two complementary arcs of $F_{r+1}$. In particular, $I'\subset B_j$. By condition \ref{condition:qs}, $f^{\circ m}|_{U_w}$ is $K$-quasiconformal. Denote by $\omega$ the inverse of $f^{\circ m}|_{U_w}$, restricted to $B_j$.  Using Theorem \ref{theorem:qcqs} as above, we conclude that $\omega$ and its inverse are $\eta$-quasisymmetric with $\eta$ as in \eqref{lemma:newelevator:eta}. In particular, $f^{\circ m}|_{I}$ is $\eta$-quasisymmetric.
    
    We finally show \eqref{lemma:newelevator:diam}. In both alternatives \ref{lemma:newelevator:1} and \ref{lemma:newelevator:2}, the map $\omega$ is $\eta$-quasisymmetric on $B_j$ and $\omega(B_j)\supset I$.  By \eqref{quasisymmetry}, we have
     $$ \frac{\diam I}{\diam \omega(B_j)}\lesssim \max\left\{ \left(\frac{\diam \omega^{-1}(I)}{\diam B_j}\right)^{K} , \left(\frac{\diam \omega^{-1}(I)}{\diam B_j}\right)^{1/K} \right\}.$$
     Note that $\diam \omega(B_j) \lesssim 1$, $\diam B_j \simeq 1$, and $\diam \omega^{-1}(I) \lesssim 1$, where $\omega^{-1}(I)=f^{\circ m}(I)$. Hence, 
     $$\diam I \lesssim (\diam f^{\circ m}(I))^{1/K}.$$
     This completes the proof. 
\end{proof}

\subsection{Diameters of dynamical and non-dynamical arcs}The next lemma supplements the properties of hyperbolic and parabolic points; see Section \ref{section:hyperbolic_parabolic_points}. 

\begin{lemma}[Diameters of consecutive arcs]\label{lemma:qs_i2}
    Suppose that $\mathcal P(f;\{a_0,\dots,a_r\})$ is a primitive Markov partition that satisfies conditions \ref{condition:uv} and \ref{condition:qs}. For each $a\in F_1$, $p\in \N$ with $p\geq 2$, and for all sufficiently large $N_0\in \N$, there exists a distortion function $\eta$ and a constant $L\geq 1$ with the following properties. If $I_1,\dots,I_p$ are consecutive complementary arcs of $F_n$, $n\geq 1$, such that $a$ is an endpoint of $I_1$, then
    $$L^{-1}\leq \frac{\diam I_i}{\diam I_j}\leq L$$
    for $i,j\in \{2,\dots,p\}$. Furthermore, if $n\geq N_0$, then the map $f^{\circ(n-N_0)}$ is $\eta$-quasisymmetric on $\bigcup_{i=2}^p I_i$. 
\end{lemma}
In essence, this lemma allows us to blow up quasisymmetrically the arcs $I_2,\dots,I_p$ to large complementary arcs of $F_{N_0}$; this type of conclusion is not provided by the quasiconformal elevator in Lemma \ref{lemma:newelevator}. Also, observe that the arc $I_1$ is intentionally excluded. In fact, we cannot guarantee that $I_1$ can be blown up quasisymmetrically to a large arc due to the presence of parabolic points.  

\begin{proof}
    Let $p\in \N$, $p\geq 2$, and consider $N_0\in \N$ so that
    $$N_0> \max\{1+\log (p+1) /\log (r+1),1+r(\log_2 (p+1)+1)\}.$$
    Also, let $n\geq N_0$ and $I_1,\dots,I_p,I_{p+1}$ be consecutive complementary arcs of $F_n$ so that $a\in F_1$ is an endpoint of $I_1$.   Suppose first that $\bigcup_{i=1}^{p+1} I_i$ is not a strict subset of a complementary arc of $F_1$. Then the number of complementary arcs of $F_n$ that are contained in a complementary arc of $F_1$ is at most $p+1$. Hence, $(r+1)^{n-1} \leq p+1$ and $n\leq 1+\log (p+1) /\log (r+1)<N_0$. This is a contradiction. 

    Thus, $I=\bigcup_{i=1}^{p+1} I_i$ is a strict subset of a complementary arc of $F_1$ having $a$ as an endpoint. Let $w$ be the word associated to $I$ and write $w=(j_1,\dots,j_l)$. Note that $n> l$ and there exists a complementary arc $A$ of $F_{l+1}$ that is contained in $I$ and has $a$ as an endpoint. By Lemma \ref{lemma:primitive}, for $i\in \N\cup \{0\}$, $A$ contains at least $2^i$ complementary arcs of $F_{l+1+ir}$. If $i$ is the largest index with $l+1+ir\leq n$, then $A$ contains at least $2^{(n-l-1)/r-1}$ complementary arcs of $F_n$. On the other hand, $A$ contains at most $p+1$ complementary arcs of $F_n$. We conclude that $n-l<1+r(\log_2(p+1)+1)$. In particular, $n-N_0<l<n$. 
    
    By condition \ref{condition:qs}, there exists a uniform constant $K\geq 1$ such that $f^{\circ (n-N_0)}|_{U_w}$ is a $K$-quasiconformal map onto $U_{j_{k}\dots j_l}$, where $k=n-N_0+1$. Note that $f^{\circ (n-N_0)}$ maps $J=\bigcup_{i=2}^p I_i$ into a fixed compact subset of $U_{j_{k}\dots j_l}$, namely into the union of all complementary arcs of $F_{N_0}$ that are contained in $\inter A_{j_k\dots j_l}$. It is important here that we exclude the first and last arcs $I_1$ and $I_{p+1}$. By Theorem \ref{theorem:qcqs}, $f^{\circ (n-N_0)}|_J$ is $\eta$-quasisymmetric for some distortion function $\eta$, depending on $U_{j_k\dots j_l}$.  Since $l-k< N_0-1$, there is a bounded number of possibilities for the region $U_{j_{k}\dots j_l}$, so $\eta$ may be chosen to be a uniform distortion function.  This proves the second conclusion of the lemma.

    The first conclusion of the lemma follows the fact that $f^{\circ (n-N_0)}|_J$ is quasisymmetric when $n\geq N_0$, combined with \eqref{quasisymmetry}, and the fact that there are finitely many complementary arcs of $F_n$ when $n\leq N_0$.
\end{proof}

The definitions of hyperbolic and parabolic points provide diameter bounds for {dynamical arcs}, i.e., complementary arcs of $F_n$. The next technical lemma provides diameter estimates for a \textit{non-dynamical arc} $I$ that is located near a point $a\in F_1$.

\begin{lemma}[Diameters of non-dynamical arcs]\label{lemma:one_sided_estimates}
Suppose that $\mathcal P(f;\{a_0,\dots,a_r\})$ is a primitive Markov partition that satisfies conditions \ref{condition:hp}, \ref{condition:uv}, and \ref{condition:qs}. For each $a\in F_1$, $p\in \N$ with $p\geq 2$, and $M\geq 1$, there exists $L\geq 1$ such that the following statement is true. If $I_1,\dots,I_p$ are consecutive complementary arcs of $F_n$, $n\geq 1$, $a$ is an endpoint of $I_1$, and $I_1 \subset \arc{[a,z_0]}$ for some $z_0\neq a$, then for each closed arc $I\subset \mathbb S^1$ with 
\begin{align*}
I\subset \bigcup_{i=1}^p I_i, \quad I\cap \bigcup_{i=2}^p I_i\neq \emptyset,\quad \textrm{and}\quad  \diam{I} \geq M^{-1} \diam{I_2}
\end{align*}
one of the following alternatives holds. 
\begin{itemize}
\item If $a^+$ is parabolic and $\alpha=1/N(a^+)$, then	
\begin{align*}
L^{-1} n^{-\alpha-1}\leq \frac{\diam{I}}{\min \{l+1,n\}}\leq L n^{-\alpha-1},
\end{align*}
where $l\in \N\cup \{0\}$ is the smallest integer, if there exists one, such that there exists a complementary arc of $F_{n+l}$ contained in $\arc{[a,z_0]}$ that has $a$ as an endpoint and does not intersect $I$; if no such integer exists, we set $l=\infty$.
\item If $a^+$ is hyperbolic, then 
$$L^{-1}\lambda(a^+)^{-n}\leq \diam{I} \leq L\lambda(a^+)^{-n}.$$ 
\end{itemize}
The corresponding estimates hold for $a^-$ if $I_1 \subset \arc{[z_0,a]}$ for some $z_0\neq a$.
\end{lemma}

The proof of Lemma \ref{lemma:one_sided_estimates} is identical to the proof of Lemma 4.20 in \cite{LyubichMerenkovMukherjeeNtalampekos:David} and we omit it. Conditions \ref{condition:uv} and \ref{condition:qs}, are only used to apply Lemma \ref{lemma:qs_i2}, which guarantees that $\diam I_i\simeq \diam I_2$ for $i\in \{2,\dots,p\}$. The initial assumptions about the relative position of $I$ and the relative size of $I$, imply that the diameter of $I$ can be estimated using the dynamical complementary arcs of $F_n$. Specifically, the estimates for $\diam I$ follow from the hyperbolic and parabolic estimates for the point $a\in F_1$, as provided by Definitions \ref{definition:hyperbolic} and \ref{definition:parabolic}. 

The next lemma guarantees that if Lemma \ref{lemma:one_sided_estimates} is applicable to an arc $I$, then it is also applicable to the image $h(I)$ under a homeomorphism conjugating $f$ to $g$.

\begin{lemma}\label{lemma:qs_i}
    Suppose that $f$ and $g$ are expansive and satisfy conditions \ref{condition:uv} and \ref{condition:qs}. Let $h\colon \mathbb S^1\to \mathbb S^1$ be a homeomorphism that conjugates $f$ to $g$ and satisfies $h(a_k)=b_k$, $k\in \{0,\dots,r\}$. For each  $a\in F_1$, $p\in \N$ with $p\geq 2$, and $M\geq 1$, there exists $L\geq 1$ such that the following statement is true. If $I_1,\dots,I_p$ are consecutive complementary arcs of $F_n$, $n\geq 1$, and $a$ is an endpoint of $I_1$, then for each closed arc $I\subset \mathbb S^1$ with 
\begin{align*}
I\subset \bigcup_{i=1}^p I_i, \quad  I\cap \bigcup_{i=2}^p I_i\neq \emptyset, \quad \textrm{and}\quad  \diam{I} \geq M^{-1} \diam{I_2}
\end{align*}
we have $\diam{h(I)} \geq L^{-1} \diam{h(I_2)}$.
\end{lemma}
    \begin{proof}
        By Lemma \ref{lemma:primitive} each of $I_1,\dots,I_p$ contains at least two complementary arcs of $F_{n+r}$. Denote by $J_i$, $i\in \{1,\dots,p'\}$, the family of consecutive complementary arcs of $F_{n+r}$ that are contained in $\bigcup_{i=1}^{p}I_i$, where $a$ is an endpoint of $J_1$ and $J_1\cup J_2\subset I_1$. Also, let $i_0\in \{3,\dots,p'\}$ such that $J_{i_0}\subset I_2$. Note that $p'\leq p(r+1)^{r-1}$. By Lemma \ref{lemma:qs_i2}, we have
        \begin{align*}
            \diam I_2 &\leq \diam\bigg( \bigcup_{i=2}^p I_i\bigg)\leq \diam  \bigg(\bigcup_{i=2}^{p'}J_i\bigg) \simeq \diam J_2 \simeq \diam J_{i_0} \lesssim \diam I_2.
        \end{align*}
        The same is true for the images under $h$, hence
        \begin{align}\label{lemma:qs_i_i2}
            \diam I_2\simeq \diam J_2 \quad \textrm{and} \quad \diam h(I_2)\simeq \diam h(J_2).
        \end{align}
        
        Let $I\subset \mathbb S^1$ be a closed arc as in the statement, in particular, satisfying $\diam{I} \geq M^{-1} \diam{I_2}$. If $I\cap J_1\neq \emptyset$, then we have $J_2\subset I$, since $I\cap \bigcup_{i=2}^p I_i\neq \emptyset$. If $I\cap J_1=\emptyset$, then $I\subset \bigcup_{i=2}^{p'}J_i$. In both cases, we have
        \begin{align}\label{lemma:qs_k2}
            \diam J\simeq \diam J_2, \quad \textrm{where} \quad J= I\cap \bigcup_{i=2}^{p'}J_i.
        \end{align}

        By Lemma \ref{lemma:qs_i2}, there exists $N_0\in \N$ and a distortion function $\eta$ such that if $n+r\geq N_0$, then $f^{\circ (n+r-N_0)}$ and $g^{\circ (n+r-N_0)}$ are $\eta$-quasisymmetric on $\bigcup_{i=2}^{p'}J_i$ and $h(\bigcup_{i=2}^{p'}J_i)$, respectively. Therefore, by \eqref{lemma:qs_k2} and \eqref{quasisymmetry} we have
        $$\diam f^{\circ (n+r-N_0)}(J) \simeq \diam f^{\circ (n+r-N_0)}(J_2).$$
        Since $f^{\circ (n+r-N_0)}(J_2)$ is a complementary arc of $F_{N_0}$, its diameter is comparable to $1$. The fact that $h$ is a homeomorphism conjugating $f$ to $g$ implies that
        $$\diam g^{\circ (n+r-N_0)}(h(J)) \simeq \diam g^{\circ (n+r-N_0)}(h(J_2))\simeq 1.$$
        Since $g^{\circ (n+r-N_0)}$ is $\eta$-quasisymmetric on $J$, $\diam h(J) \simeq \diam h(J_2)$. As $I\supset J$, we have $\diam h(I) \gtrsim \diam h(J_2)$. Finally, by \eqref{lemma:qs_i_i2}, we have $\diam h(J_2) \simeq \diam h(I_2)$. This completes the proof in the case that $n+r\geq N_0$. If $n+r<N_0$, then the desired conclusion follows immediately from the uniform continuity of $h$.       
    \end{proof}

\subsection{Distortion estimates and Proof of Theorem \ref{theorem:extension_generalization}}
Let $f,g$ be as in Theorem \ref{theorem:extension_generalization}. The expansivity  and Lemma \ref{lemma:expansive_conjugate_fg} imply that the map $h\colon \{a_0,\dots,a_r\}\to \{b_0,\dots,b_r\}$ from the statement of Theorem \ref{theorem:extension_generalization} extends to an orientation-pre\-serving homeomorphism $h$ of $\mathbb S^1$ that conjugates $f$ to $g$.

Let $I,J\subset \mathbb S^1$ be adjacent closed arcs each of which has length $t\in (0,1/2)$.  Consider points $a\in F_1$ and $b=h(a)$ arising by  applying Lemma \ref{lemma:newelevator} to the arc $I\cup J$.  We will show that 
\begin{align}\label{distortion:qs}
    \diam h(I) \simeq \diam h(J) \quad \textrm{under condition \ref{HH}}
\end{align}
for the points $a,b$ and that
\begin{align}\label{distortion:david}
    \max \left\{\frac{\diam h(I)}{\diam h(J)}, \frac{\diam h(J)}{\diam h(I)} \right\}\lesssim \log(1/t)\quad \textrm{under condition \ref{HP}}
\end{align}
for $a,b$. The above estimates, when combined with the Beurling--Ahlfors extension theorem \cite{BeurlingAhlfors:extension}, which provides a quasiconformal extension, or with Theorem \ref{theorem:extension_david}, which provides a David extension,  complete the proof of Theorem \ref{theorem:extension_generalization}.

Let $m$ be as in Lemma \ref{lemma:newelevator}, so that $f^{\circ m}|_{I\cup J}$ and $g^{\circ m}|_{h(I\cup J)}$ are $\eta$-quasisymmetric for some uniform distortion function $\eta$. We let $I'=f^{\circ m}(I)$ and $J'=f^{\circ m}(J)$. Since $\diam I=\diam J$, we have $\diam{I'}\simeq \diam{J'}$. We consider two main cases for the proof, depending on which of the two alternatives of Lemma \ref{lemma:newelevator} applies to $I\cup J$.

\subsubsection{Alternative \ref{lemma:newelevator:1}}
Suppose that alternative \ref{lemma:newelevator:1} holds for $I\cup J$. Then there exists $p\in \N$ that is uniformly bounded above, $n\in \N\cup \{0\}$, and consecutive complementary arcs $I_1',\dots,I_p'$ of $F_{s}$, where $s=n+r+1$, such that $a$ is an endpoint of $I_1'$ and
\begin{align}\label{rel_pos}
    I_{i_0}'\subset I'\cup J'\subset \bigcup_{i=1}^p I_i'
\end{align}
for some $i_0\in \{2,\dots,p-1\}$. Without loss of generality, suppose that $\bigcup_{i=1}^p I_i'\subset \arc{[a,z_0]}$ for some $z_0\neq a$, and that $I'$ is closer to $a$ than $J'$ within $\bigcup_{i=1}^p I_i'$. The relative position of $I'$ and $J'$ implies that $J'\cap \bigcup_{i=2}^p I_i'\neq \emptyset$; see Figure \ref{fig:distortion:ai}.

Let $k\in \N\cup \{0\}$ be the smallest integer such that a complementary arc of $F_{s+k}$ not having $a$ as an endpoint intersects $I'$. By the definition of $k$ there exist consecutive complementary arcs $J_1',\dots,J_{p'}'$ of $F_{s+k}$, where $p'\leq \max\{r+1,p\}$, such that $I'\subset \bigcup_{i=1}^{p'}J_i'$ and $I'\cap \bigcup_{i=2}^{p'}J_i'\neq \emptyset$; see Figure \ref{fig:distortion:ai}. Also, let $l_1\in \N\cup \{0,\infty\}$ (resp.\ $l_2\in \N\cup \{0\}$) be the smallest number such that there exists a complementary arc of $F_{s+k+l_1}$ (resp.\ $F_{s+l_2}$) contained in $\arc{[a,z_0]}$ that has $a$ as an endpoint and does not intersect $I'$ (resp.\ $J'$). Observe that $k\leq l_2\leq k+1$. 

\begin{figure}
    \centering
    \begin{tikzpicture}
        \draw (0,0)--(8,0);
        \fill (1,0) circle (2pt) node[anchor=north] {$a$};
        \draw[line width=2pt] (2,0)--(7,0);
        \draw (4,0.07)--(4,-0.07);
        \draw [decorate, decoration = {brace,mirror}] (2,-0.3) --  (3.95,-0.3) node[pos=0.5, anchor=north]{$I'$};
        \draw [decorate, decoration = {brace,mirror}] (4.05,-0.3) --  (7,-0.3) node[pos=0.5, anchor=north]{$J'$};

        \draw [decorate, decoration = {brace}] (1.,0.8) --  (4.5,0.8) node[pos=0.5, anchor=south]{$I_1'$};
        \draw [decorate, decoration = {brace}] (4.55,0.8) --  (6,0.8) node[pos=0.5, anchor=south]{$I_2'$};
        \draw [decorate, decoration = {brace}] (6.05,0.8) --  (8,0.8) node[pos=0.5, anchor=south]{$I_3'$};
        \draw [decorate, decoration = {brace}] (1.,0.2) --  (2.5,0.2) node[pos=0.5, anchor=south]{$J_1'$};
        \draw [decorate, decoration = {brace}] (2.55,0.2) --  (4.5,0.2) node[pos=0.5, anchor=south]{$J_2'$};
    \end{tikzpicture}   
    \caption{Relative positions of arcs in the case of \ref{lemma:newelevator:1}.}
    \label{fig:distortion:ai}
\end{figure}

\medskip
\noindent
\textbf{Case {H}:} $a^+$ is hyperbolic.  By Lemma \ref{lemma:one_sided_estimates}, $\diam J' \simeq \lambda(a^+)^{-s}$ and $\diam I'\simeq \lambda(a^+)^{-s-k}$. Since $\diam I'\simeq \diam J'$, we conclude that $k\lesssim 1$. Therefore,
\begin{align}\label{distortion:ai:hyperbolic}
    \diam I'\simeq \diam J' \simeq \lambda(a^+)^{-s}.
\end{align}
By Definition \ref{definition:hyperbolic}, we have $\diam J_2' \simeq \lambda(a^+)^{-s-k} \simeq \diam I'$.  Therefore, Lemma \ref{lemma:qs_i} is applicable, and 
\begin{align}
    \diam h(J')\gtrsim \diam h(I_2')\quad \textrm{and} \quad \diam h(I') \gtrsim \diam h(J_2'). 
\end{align}
These inequalities imply that Lemma \ref{lemma:one_sided_estimates} is applicable to the images $h(I')$, $h(J')$. 

\medskip
\noindent
\textbf{Case {H}$\to${H}:} $b^+$ is hyperbolic. By Lemma \ref{lemma:one_sided_estimates},
$$ \diam h(I')\simeq \lambda(b^+)^{-s-k}\simeq \lambda(b^+)^{-s}\simeq  \diam h(J').$$
Since $g^{\circ m}|_{h(I\cup J)}$ is quasisymmetric, \eqref{distortion:qs} follows.

\medskip
\noindent
\textbf{Case {H}$\to${P}:} $b^+$ is parabolic. By \eqref{lemma:newelevator:diam} in Lemma \ref{lemma:newelevator}, we have $$\diam I' \gtrsim (\diam I)^{\gamma} \simeq t^\gamma.$$
Combined with \eqref{distortion:ai:hyperbolic} and the assumption that $t\in (0,1/2)$, this gives $\log(1/t)\gtrsim s\simeq n+1$.  By Lemma \ref{lemma:one_sided_estimates}, combined with the facts that $k\lesssim 1$ and $k\leq l_2\leq k+1$, we have
\begin{align*}
    \diam h(I') &\simeq (s+k)^{-\beta-1}\min\{ l_1+1,s+k\} \simeq s^{-\beta -1} \min\{ l_1+1,s\}  \quad \textrm{and}\\
    \diam h(J') &\simeq s^{-\beta-1} \min\{l_2+1,s\} \simeq s^{-\beta-1} \min\{1,s\};
\end{align*}
here $\beta=1/N(b^+)$. Thus, 
$$ \max \left\{\frac{\diam h(I')}{\diam h(J')}, \frac{\diam h(J')}{\diam h(I')} \right\}\lesssim s \simeq n+1.$$
By the final part of \ref{lemma:newelevator:1} and assuming condition \ref{condition:qs_strong} for the map $g$, we see that $g^{\circ m}|_{h(I\cup J)}$ is $\eta_n$-quasisymmetric.  By \eqref{quasisymmetry}, we conclude that, for $K_n=1+M(1+\log (n+1))^{-1}$, we have
$$ \max \left\{\frac{\diam h(I)}{\diam h(J)}, \frac{\diam h(J)}{\diam h(I)} \right\}\lesssim (n+1)^{K_n} \simeq n+1 \lesssim \log(1/t).$$
Thus, \eqref{distortion:david} is true. 

\medskip
\noindent
\textbf{Case P$\to$P:} $a^+$ and $b^+$ are parabolic. Let $\alpha=1/N(a^+)$ and $\beta=1/N(b^+)$. By Lemma \ref{lemma:one_sided_estimates}, since $k\leq l_2\leq k+1$, we have
\begin{align*}
    s^{-\alpha-1}\min\{k+1,s\} \simeq \diam J' \simeq \diam I' \simeq (s+k)^{-\alpha-1}\min\{l_1+1,s+k\}.
\end{align*}
If $k+1\geq s$, we obtain
$$s^{-\alpha} \simeq (k+1)^{-\alpha-1}\min\{l_1+1,s+k\}\lesssim (k+1)^{-\alpha-1}(s+k) \simeq (k+1)^{-\alpha}.$$
We conclude that $k+1\lesssim s$. In any case, $0\leq k\lesssim s$. 

By Lemma \ref{lemma:one_sided_estimates} and Definition \ref{definition:parabolic},
$$ \diam I' \simeq (s+k)^{-\alpha-1} \min\{l_1+1,s+k\} \gtrsim (s+k)^{-\alpha-1} \simeq \diam J_2'.$$
Therefore, Lemma \ref{lemma:qs_i} can be applied to $I'$, so $\diam h(I')\gtrsim \diam h(J_2')$. The relative position of $I',J'$ and \eqref{rel_pos} imply that Lemma \ref{lemma:qs_i} applies to $J'$ as well, so  $\diam h(J')\gtrsim \diam h(I_2')$. We now apply Lemma \ref{lemma:one_sided_estimates} to $h(I')$ and $h(J')$ and obtain
\begin{align*} \frac{\diam h(J')}{\diam J'} &\simeq \frac{s^{-\beta-1} \min\{l_2+1,s\}}{s^{-\alpha-1} \min\{l_2+1,s\}} \simeq s^{\alpha- \beta}, \quad \textrm{and}\\
    \frac{\diam h(I')}{\diam I'}&\simeq  \frac{(s+k)^{-\beta-1} \min\{l_1+1,s+k\}}{(s+k)^{-\alpha-1}\min\{l_1+1,s+k\}} \simeq (s+k)^{\alpha-\beta} \simeq s^{\alpha-\beta}.
\end{align*}
Altogether, since $\diam I'\simeq \diam J'$, we have $\diam h(I')\simeq \diam h(J')$. The fact that $g^{\circ m}|_{h(I\cup J)}$ is quasisymmetric implies \eqref{distortion:qs}.

\subsubsection{Alternative \ref{lemma:newelevator:2}}
Suppose that alternative \ref{lemma:newelevator:2} holds for $I\cup J$. Without loss of generality assume that $J'\subset \arc{[a,z_0]}$ for some $z_0\neq a$ and $I'=I_1'\cup I_2'$, where $I_2'\subset \arc{[a,z_0]}$ and $I_1'\subset \arc{[z_0,a]}$. According to \ref{lemma:newelevator:2}, $I_2'\cup J'$ satisfies the conclusions of \ref{lemma:newelevator:1}. In particular, there exist $p\in \N$, uniformly bounded above, $n \in \N\cup \{0\}$, and consecutive complementary arcs $J_1',\dots, J_p'$ of $F_s$, where $s=n+r+1$, such that $a$ is an endpoint of $J_1'$ and 
\begin{align*}
    J_{2}'\subset I_2'\cup J' \subset \bigcup_{i=1}^p J_i'.
\end{align*}
See Figure \ref{fig:distortion:aii} for an illustration. For $i=1,2$, denote by $w_i$ the word associated to $I_i'$ and set $s_i=|w_i|+1$.  Let $l\in \N\cup \{0\}$ be the smallest integer such that there exists a complementary arc of $F_{s+l}$ contained in $\arc{[a,z_0]}$ that has $a$ as an endpoint and does not intersect $J'$. Observe that $s_2=s+l$.

\begin{figure}
    \centering
    \begin{tikzpicture}
        \draw (-1.5,0)--(8,0);
        \fill (1,0) circle (2pt) node[anchor=north] {$a$};
        \draw[line width=2pt] (-0.5,0)--(7,0);
        \draw (4,0.07)--(4,-0.07);
        \draw [decorate, decoration = {brace,mirror}] (1.05,-0.3) --  (3.95,-0.3) node[pos=0.5, anchor=north]{$I_2'$};
        \draw [decorate, decoration = {brace,mirror}] (-0.5,-0.3) --  (0.95,-0.3) node[pos=0.5, anchor=north]{$I_1'$};
        \draw [decorate, decoration = {brace,mirror}] (4.05,-0.3) --  (7,-0.3) node[pos=0.5, anchor=north]{$J'$};

        \draw [decorate, decoration = {brace}] (1.,0.3) --  (4.5,0.3) node[pos=0.5, anchor=south]{$J_1'$};
        \draw [decorate, decoration = {brace}] (4.55,0.3) --  (6,0.3) node[pos=0.5, anchor=south]{$J_2'$};
        \draw [decorate, decoration = {brace}] (6.05,0.3) --  (8,0.3) node[pos=0.5, anchor=south]{$J_3'$};
    \end{tikzpicture}   
    \caption{Relative positions of arcs in the case of \ref{lemma:newelevator:2}.}
    \label{fig:distortion:aii}
\end{figure}
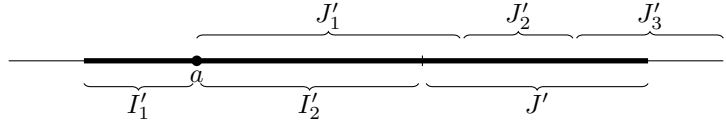

\medskip
\noindent 
\textbf{Case \ref{HH}:} There exists $\mu>0$ such that if $a^{\pm}$ is hyperbolic, then $b^{\pm}$ is also hyperbolic with $\lambda(a^\pm)^{\mu} =\lambda(b^{\pm})$ and if $a^{\pm}$ is parabolic, then $b^{\pm}$ is also parabolic with $\mu^{-1}N(a^{\pm})=N(b^{\pm})$. Our goal is to show that in all of these cases we have
\begin{align}\label{distortion:aii:main_claim}
    \diam h(J')\simeq (\diam J')^{\mu}\quad \textrm{and} \quad \diam h(I_i') \simeq (\diam I_i')^{\mu}
\end{align}
for $i=1,2$. These conditions imply that $\diam h(I')\simeq \diam h(J')$, so \eqref{distortion:qs} follows since the map $g^{\circ m}|_{h(I\cup J)}$ is a quasisymmetry.

\medskip
\noindent
\textbf{Case H$\to$H:} $a^+$ and $b^+$ are hyperbolic. By Definition \ref{definition:hyperbolic}, we have
\begin{align*}
    \diam I_2' \simeq \lambda(a^+)^{-s_2} \,\,\, \textrm{and}\,\,\,  \diam h(I_2') \simeq \lambda(b^+)^{-s_2}
\end{align*}
so $\diam h(I_2')\simeq (\diam I_2')^{\mu}$. Moreover, by Lemma \ref{lemma:one_sided_estimates}, we have
$$\diam J' \simeq \lambda(a^+)^{-s} \simeq \diam J_2'.$$
Thus, we apply Lemma \ref{lemma:qs_i} to $J'$ and obtain $\diam h(J') \gtrsim \diam h(J_2')$.  By Lemma \ref{lemma:one_sided_estimates} we now have
\begin{align*}
    \diam h(J') &\simeq \lambda(b^+)^{-s} \simeq (\diam J')^{\mu}.
\end{align*}
If $a^-$ and $b^-$ are hyperbolic, then with the same argument we have $\diam h(I_1') \simeq \diam  (I_1')^{\mu}$.

\medskip
\noindent
\textbf{Case P$\to$P:} $a^+$ and $b^+$ are parabolic. Let $\alpha^+=1/N(a^+)$ and $\beta^+=1/N(b^+)$, so $\mu \alpha^+=\beta^+$. By Definition \ref{definition:parabolic},
\begin{align}\label{distortion:aii:par_par}
    \diam I_2' \simeq s_2^{-\alpha^+} \quad \textrm{and}\quad \diam h(I_2') \simeq s_2^{-\beta^+}.
\end{align}
Hence, $\diam h(I_2')\simeq (\diam I_2')^{\mu}$. By Lemma \ref{lemma:one_sided_estimates}, we have
\begin{align*}
    \diam J'\simeq s^{-\alpha^+-1}\min\{l+1,s\}\quad \textrm{and}\quad \diam h(J') \simeq s^{-\beta^+-1} \min\{l+1,s\}.
\end{align*}
Given that $s_2=s+l$, by \eqref{distortion:aii:par_par} we obtain 
$$(s+l)^{-\alpha^+} \simeq \diam I_2' \lesssim \diam I' \simeq \diam J'\simeq   s^{-\alpha^+-1}\min\{l+1,s\},$$
which implies that $l+1\gtrsim s$. Therefore, 
$$\diam J' \simeq s^{-\alpha^+} \quad \textrm{and}\quad \diam h(J') \simeq s^{-\beta^+},$$
which give the desired $\diam h(J')\simeq (\diam J')^{\mu}$. If $a^-$ and $a^-$ are parabolic then with the same argument as in $I_2'$ we obtain $\diam h(I_1')\simeq (\diam I_1')^{\mu}$. We have completed the verification of \eqref{distortion:aii:main_claim}.

\medskip
\noindent 
\textbf{Case \ref{HP}:} $a$ is symmetrically hyperbolic and $b$ is symmetrically parabolic. We set $\lambda=\lambda(a)$ and $\beta=1/N(b)$. We can apply Lemma \ref{lemma:one_sided_estimates} and obtain
\begin{align*}
\diam{J'}\simeq \lambda^{-s} \quad &\textrm{and}\quad  \diam{h(J')} \simeq s^{-\beta-1} \min \{l+1,s\}, \\
\diam{I_2'} \simeq \lambda^{-s_2} \quad &\textrm{and}\quad \diam{h(I_2')} \simeq s_2^{-\beta}, \quad \textrm{and}\\
\diam{I_1'} \simeq \lambda^{-s_1} \quad &\textrm{and}\quad \diam{h(I_1')} \simeq s_1^{-\beta}.
\end{align*}
By \eqref{lemma:newelevator:diam} in Lemma \ref{lemma:newelevator}, we have $$\diam J' \gtrsim (\diam J)^{\gamma} \simeq t^\gamma.$$
Combined with the above and the assumption that $t\in (0,1/2)$, this gives $\log(1/t)\gtrsim s\simeq  n+1$. 

If $\diam I_2'\leq \diam I_1'$, then $\diam I_1'\simeq \diam J'$. Hence, $s_2\gtrsim s_1$ and $s_1\simeq s$. These conditions give $\diam h(I') \simeq s_1^{-\beta}\simeq s^{-\beta}$. If, instead, $\diam I_1'\leq \diam I_2'$, the same analysis gives $\diam h(I') \simeq s_2^{-\beta}\simeq s^{-\beta}$. In both cases, we have
\begin{align}\label{distortion:hp}
    \diam h(I') \simeq  \diam h(J_1') \simeq s^{-\beta}.
\end{align}
We conclude that
\begin{align*}
    \frac{\diam h(J')}{\diam h(J_1')} \simeq s^{-1}\min\{l+1,s\} \,\,\, \textrm{and}\,\,\, \max \left\{\frac{\diam h(J_1')}{\diam h(J')}, \frac{\diam h(J')}{\diam h(J_1')} \right\}\lesssim s.
\end{align*}
Under condition \ref{condition:qs_strong}, the map $g^{\circ m}|_{h(J_1'\cup J')}$ is $\eta_n$-quasisymmetric, so by \eqref{quasisymmetry}
$$ \max \left\{\frac{\diam h(J_1)}{\diam h(J)}, \frac{\diam h(J)}{\diam h(J_1)} \right\}\lesssim s^{K_n}\simeq  (n+1)^{K_n}\simeq n+1 \lesssim \log(1/t).$$
Finally, $g^{\circ m}|_{h(I\cup J)}$ is quasisymmetric, so \eqref{distortion:hp} gives $\diam h(J_1)\simeq \diam h(I)$. Thus,
$$ \max \left\{\frac{\diam h(I)}{\diam h(J)}, \frac{\diam h(J)}{\diam h(I)} \right\}\lesssim \log(1/t),$$
as desired in \eqref{distortion:david}. This completes the proof.

\section{Classification of piecewise quasiconformal circle maps}

\begin{theorem}\label{theorem:classification}
    Let $f\colon \mathbb S^1\to \mathbb S^1$ be an expansive covering map with a Markov partition $\mathcal P(f;\{a_0,\dots,a_r\})$ such that
    \begin{enumerate}[label=\normalfont(\roman*)]
        \item conditions \ref{condition:hp}, \ref{condition:uv}, and \ref{condition:qs} are satisfied,
        \item\label{class:ii} the restriction of $f$ to the arc $A_k=\arc{[a_k,a_{k+1}]}$ is injective and at most one of the endpoints of $A_k$ is periodic for each $k\in \{0,\dots,r\}$, and 
        \item each point of $\{a_0,\dots,a_r\}$ is either symmetrically hyperbolic or symmetrically parabolic. 
    \end{enumerate}
     Then there exists an expansive covering map $g\colon \mathbb S^1\to \mathbb S^1$ with a Markov partition $\mathcal P(g;\{b_0,\dots,b_r\})$ and an orientation-preserving homeomorphism $h\colon \mathbb S^1\to \mathbb S^1$ that conjugates $f$ to $g$ with $h(a_k)=b_k$ for each $k\in \{0,\dots,r\}$ and such that
     \begin{enumerate}[label=\normalfont(\Roman*)]
        \item conditions \ref{condition:hp}, \ref{condition:uv}, and \ref{condition:qs} are satisfied by $\mathcal P(g;\{b_0,\dots,b_r\})$,
        \item\label{class:II} the restriction of $g$ to the arc $B_k=\arc{[b_{k},b_{k+1}]}$ is the restriction of a M\"obius or anti-M\"obius transformation $M_k$ (depending on the orientation of $f$) of the unit disk for each $k\in \{0,\dots,r\}$, and
        \item\label{class:III} for each $a\in \{a_0,\dots,a_r\}$ and for $b=h(a)$, we can prescribe whether $b$ is symmetrically hyperbolic or symmetrically parabolic.  
    \end{enumerate}
\end{theorem}

Recall that the indices are taken in modulo $r+1$; e.g.\ $a_{r+1}=a_0$ and $a_{-1}=a_r$.  

\begin{remark}\label{remark:classification}
If in \ref{class:III} we prescribe that $b=h(a)$ is symmetrically hyperbolic (resp.\ parabolic) if and only if $a$ is symmetrically hyperbolic (resp.\ parabolic) for each $a\in \{a_0,\dots,a_r\}$, then by Theorem \ref{theorem:extension_generalization} the conjugacy $h$ is quasisymmetric. Hence, Theorem \ref{theorem:classification} provides a classification of piecewise quasiconformal circle maps up to quasisymmetric conjugacy. 
\end{remark}

Before proving the theorem, we establish a general criterion for the expansivity of a map.

\begin{theorem}\label{theorem:expansive_criterion}
    Let $f\colon \mathbb S^1\to \mathbb S^1$ be a covering map and $\mathcal P(f;\{a_0,\dots,a_r\})$ be a primitive Markov partition satisfying conditions \ref{condition:uv} and \ref{condition:qs}. Suppose that $f$ is expansive at each periodic point of $\{a_0,\dots,a_r\}$. Then $f$ is expansive. 
\end{theorem}

\begin{proof}
By replacing $f$ with $f\circ f$ and changing appropriately the Markov partition, we may assume that $f$ is orientation-preserving; see \ref{expansive:iterate_iff}.

Suppose that $f$ is not expansive. By \ref{expansive:diameters} there exists $\delta>0$, $j_i\in \{0,\dots,r\}$, $i\in \N$, and $w(n)=(j_1,\dots,j_n)$ such that $\diam A_{w(n)}\to \delta$ as $n\to\infty$. We will first reduce to the case that there exists a point $a\in F_1$ and a nested sequence of complementary arcs of $F_n$ that have $a$ as an endpoint and have large diameter. 

Since the Markov partition is primitive, by Lemma \ref{lemma:primitive}, there exists a subsequence $\{w(k_n)\}_{n\in \N}$ of $\{w(n)\}_{n\in \N}$ such that $J_n=A_{w(k_n)}$ is a complementary arc of $F_{|w(k_n)|}$ but not of $F_{|w(k_n)|+1}$, $n\in \N$. Thus, $w(k_n)$ is the word associated to $J_n$. By Lemma \ref{lemma:newelevator} \ref{lemma:newelevator:1} applied to $J_n$, if we consider the canonical splitting $w(k_n)=(v(n),u(n))$ and $m(n)=|v(n)|$, then $A_{u(n)}$ has a point $a(n)$ of $F_1$ as an endpoint and $f^{\circ m(n)}|_{J_n}\colon J_n\to A_{u(n)}$ is $\eta$-quasisymmetric for some uniform distortion function $\eta$. Moreover, there exist consecutive complementary arcs $I_1(n),\dots,I_{p_n}(n)$, $3\leq p_n\leq M$, of $F_{|w(k_n)|+r+1}$ such that $J_n=\bigcup_{i=1}^{p_n} I_i(n)$.  The notation is such that $I_1(n)$ is the arc whose image $I_1'(n)$ under $f^{\circ m(n)}$ has $a(n)$ as an endpoint. We also denote by $I(n)$ the arc among $I_1(n),\dots,I_{p_n}(n)$ that contains all but finitely many arcs $J_l$, $l\in \N$, and hence $\diam I(n)\to \delta$ as $n\to\infty$. In addition, by \eqref{lemma:newelevator:diam} in Lemma \ref{lemma:newelevator}, there exists $\delta'>0$ such that  $\diam A_{u(n)} \geq \delta'$ for all $n\in \N$.

By Lemma \ref{lemma:qs_i2}, the images of $I_2(n),\dots,I_{p_n}(n)$ under $f^{\circ m(n)}$ have comparable diameters. Since $f^{\circ m(n)}|_{J_n}$ is $\eta$-quasisymmetric, by \eqref{quasisymmetry} we conclude that $I_2(n),\dots,I_{p_n}(n)$ have comparable diameters. Given that $\diam I(n)\to \delta$ and $\diam J_n\to\delta$ as $n\to\infty$, we conclude that $I_1(n)=I(n)$ for all sufficiently large $n\in \N$. Moreover, $\diam (J_n\setminus I_1(n))\to 0$ as $n\to\infty$. Hence, by \eqref{quasisymmetry}, we have $\diam ( A_{u(n)}\setminus I_1'(n)) \to 0$. This shows, that $|u(n)|\to\infty$ as $n\to\infty$.

Thus, there exists  $a\in F_1$ and for each $n\in \N$ a complementary arc $B_n$ of $F_n$ having $a$ as an endpoint such that $B_{n+1}\subset B_n$ and $\diam B_n \geq \delta'$ for each $n\in \N$. By applying finitely many iterates of $f$, and using the uniform continuity of $(f|_{B_n})^{-1}$, we may assume that $a$ is a periodic point with period $p$. Let $C_1=B_1$ and $C_{n}=B_{np+1}$, so that $f^{\circ p}$ maps $C_{n+1}$ onto $C_{n}$.

By \ref{condition:uv} and \ref{condition:qs}, for each $n\in \N$ there exists a region $W_{n}$ containing $\inter C_n$ such that $W_{n+1}\subset W_n$ and $f^{\circ p}|_{W_{n+1}}\colon W_{n+1}\to W_n$ is $K$-quasiconformal for a uniform constant $K\geq 1$. Consider the inverse map $g_n\colon W_{n}\to W_{n+1}$. We also define $G_n=g_n\circ \dots\circ g_1\colon W_{1}\to W_{n+1}$, which is $K$-quasiconformal by \ref{condition:qs}. Since $W_{n}\subset W_1$ for each $n\in \N$, $G_n$ omits all values in the complement of $W_1$. By \cite{LehtoVirtanen:quasiconformal}*{Theorem II.5.2}, $G_n$ converges locally uniformly, after passing to a subsequence, to a constant map or to a $K$-quasiconformal embedding $G\colon W_1\to \C$.  By Lemma \ref{lemma:expansive_point}, we conclude that $G$ is the constant map $z\mapsto a$.

For $n\in \N$, let $b_n$ be the endpoint of $C_n$ that is different from $a$ and observe that $f^{\circ p}(b_{n+1})=b_n$, $n\in \N$. Since $\diam C_n\geq \delta'$ for each $n\in \N$, $b_n$ converges to a point $b\neq a$. Since the Markov partition is primitive, by Lemma \ref{lemma:primitive} we conclude that $b_{k_0}\in W_1$ for some $k_0\in \N$. Note that as $n\to\infty$, the points $G_n(b_{k_0})= b_{k_0+n}$ converge to $G(b_{k_0})=a$. Thus, $a=b$, a contradiction. 
\end{proof}

\begin{proof}[Proof of Theorem \ref{theorem:classification}]
    Suppose that $a_k$ is a periodic point of $f$. By assumption \ref{class:ii}, $f|_{A_k}$ is injective and the points $a_{k-1},a_{k+1}$ are not periodic. If $a_{k-1}=a_{k+1}$, then we necessarily have $f(a_{k+1})=a_k$, which contradicts injectivity. Thus, $a_{k-1}\neq a_{k+1}$. This implies that there are at least 3 points in the Markov partition $\{a_0,\dots,a_r\}$. 
    
    Consider points $b_0,\dots,b_r\in \mathbb S^1$ with the same cyclic order as the points $a_0,\dots,a_r$. We require that the length of $B_k$ is less than $\pi$ for $i\in \{0,\dots,r\}$; this is possible since there are at least 3 arcs. Define $g(b_i)=b_j$ whenever $f(a_i)=a_j$, $i,j\in \{0,\dots,r\}$. For each $k\in \{0,\dots,r\}$ we will define  $g|_{B_k}$ to be an appropriate (anti-)M\"obius transformation such that $B_i\subset g(B_k)$ if and only  $A_i\subset f(A_k)$. 

    For simplicity, suppose that $f$ is orientation-preserving. Let $k\in \{0,\dots,r\}$ and suppose that $f$ maps the arc $A_k$ onto an arc $\arc{[a_i,a_j]}$. By assumption, $f|_{A_k}$ is injective, so $a_i\neq a_j$. We consider two cases.

    \medskip
    \noindent   
    \textbf{Case 1:} The endpoints of $A_k$ are not periodic. We define $g|_{B_k}$ to be the restriction of a M\"obius transformation $M_k$ of the unit disk that maps the arc $B_k$ onto $\arc{[b_i,b_j]}$. 

    \medskip
    \noindent   
    \textbf{Case 2:} One of the endpoints of $A_k$, say $a_k$, periodic. By assumption, $a_{k+1}$ is not periodic.  If $b_k$ is prescribed in \ref{class:III} to be symmetrically hyperbolic, we define $g|_{B_k}$ to be the restriction of a M\"obius transformation  $M_k$ of the unit disk that maps $B_k$ to $\arc{[b_i,b_j]}$ and such that $(g|_{B_k})'(b_k)=2>1$. If, instead $b_k$ is prescribed to be symmetrically parabolic, then we require that $g|_{B_k}$ is a M\"obius transformation $M_k$ of the unit disk that maps $B_k$ to $\arc{[b_i,b_j]}$ such that $(g|_{B_k})'(b_k)=1$ and $B_k$ defines a repelling direction. 

    \medskip

    By construction, $g$ is a covering map that is expansive at each periodic point of $\{b_0,\dots,b_r\}$. Consider the Markov partition $\mathcal P(g;\{b_0,\dots,b_r\})$. Since the corresponding Markov partition for $f$ is primitive, we conclude that the same is true for $g$. By Theorem \ref{theorem:expansive_criterion}, we conclude that $g$ is expansive on $\mathbb S^1$, once we verify conditions \ref{condition:uv} and \ref{condition:qs}.

    Let $D\subset \C$ be an open disk that contains for each $k\in \{0,\dots,r\}$ the planar disk that is orthogonal to the unit circle at the points $b_k$ and $b_{k+1}$; here we use the assumption that the length of $B_k$ is less than $\pi$. For $k\in \{0,\dots, r\}$, let $V_k\subset \C$ be the intersection of $D$ with a disk in the sphere $\widehat\C$ that is orthogonal to the unit circle at the points $M_k(b_k)$ and $M_k(b_{k+1})$, and such that $M_k(B_k) \subset \overline{V_k}$. Then define $U_k=M_k^{-1}(V_k)$, which is contained in the planar disk that is orthogonal to the unit circle at $b_k$ and $b_{k+1}$. Observe that $U_k\subset D$ by the choice of $D$. Moreover, if $B_i\subset g(B_k)$ then, by construction, $U_i$ is contained in $V_k$, so \ref{condition:uv} holds. Condition \ref{condition:qs} is automatic. 

    Finally, we verify condition \ref{condition:hp} and show that each point of $\{b_0,\dots,b_r\}$ is symmetrically hyperbolic or symmetrically parabolic. Since $g\colon \mathbb S^1\to \mathbb S^1$ is bi-Lipschitz, it suffices to consider periodic points. Let $b_k$ be a periodic point with period $p$. Suppose that $b_k$ is prescribed in \ref{class:III} to be symmetrically hyperbolic. Then there exists a neighborhood $W$ of $b_k$ such that $g^{\circ p}|_{B_k\cap W}$ and $g^{\circ p}|_{B_{k-1}\cap W}$ are restrictions of two hyperbolic M\"obius transformations with a repelling point at $b_k$ and with the same multiplier, equal to $2^p$. We now use \cite{LyubichMerenkovMukherjeeNtalampekos:David}*{Lemma 4.17}, which implies that $b_k$ is indeed symmetrically hyperbolic.   
    
    Suppose that $b_k$ is prescribed to be symmetrically parabolic. By construction, there exists a neighborhood $W$ of $b_k$ such that $g^{\circ p}|_{B_k\cap W}$ (resp.\ $g^{\circ p}|_{B_{k-1}\cap W}$) is the restriction of a parabolic M\"obius transformation such that $B_k$ (resp.\ $B_{k-1}$) defines a repelling direction for the fixed point $b_k$. The parabolic multiplicity for both transformations is $2$. We now use \cite{LyubichMerenkovMukherjeeNtalampekos:David}*{Lemma 4.18}, which implies that $b_k$ is indeed symmetrically parabolic.   
\end{proof}

\section{Cusps and obstructions for David homeomorphisms}\label{section:cusp}

In this section we prove the first part of Theorem \ref{theorem:basins}. The main result in \cite{IOZ21} implies that there exists no David homeomorphism of the sphere that maps a Jordan region that is smooth except at a quadratic cusp to the unit disk. This result does not apply immediately in our setting because the boundary of a parabolic basin is not necessarily smooth. Our proof is by contradiction and follows similar length-area estimates as in \cite{IOZ21}. 

\subsection{Parabolic estimates}
Let $R$ be a rational map. Let $\Omega$ be an immediate basin of a parabolic fixed point $a$ of multiplicity $2$ such that $R'(a)=1$.  After normalizing, we assume that $\overline \Omega \subset \C$, $a = 0$ is a fixed point, and $R(z) = z+z^2+O(z^3)$ in a neighborhood of $0$. We also assume that $\Omega$ is a Jordan region.

We introduce a coordinate change $w= M(z) = z^{-1}$, which sends the fixed point $0$ to $\infty$. We fix a large $C>0$ and real values $y^-<y^+$ to be specified later. Let $V = M^{-1}(\{w \colon \Re(w) > C\})$, $\alpha^\pm= M^{-1}(\{x+iy^\pm\colon x>C\})$, and $W \subset V$ be the region bounded by $\alpha^\pm$. Note that $\alpha^{\pm}$ are parabolas meeting at $0$. Also, for $t>0$ define $\widetilde J_t= \{\frac{1}{t}+iy: y\in \R\}$ and $J_t=M^{-1}(\widetilde J_t)$.   See Figure \ref{fig:parabolic} for an illustration. The next lemma is independent of the dynamics. 
\begin{lemma}\label{lem:compM}
     For each Borel function $\rho\colon W\to [0,\infty]$ we have
    $$ \int_0^{C^{-1}} \int_{J_t\cap W} \rho \, d\mathcal H^1 dt \simeq  \int_W \rho.$$
\end{lemma}
\begin{proof}
    Consider the projection map $G\colon V \to (0, C^{-1})$ defined by sending a point $z \in J_t$ to $t$.
Note that $G(z) = \frac{1}{\Re(\frac{1}{z})} = x + \frac{y^2}{x}$, where $z = x+iy$. Since $W$ is bounded by the parabolas $\alpha^\pm$, we have $y = O(x^2)$ on $W$. Thus, 
$$\nabla G = \begin{bmatrix}
    1 - \frac{y^2}{x^2}, & \frac{2y}{x}
\end{bmatrix} = \begin{bmatrix}
    1+O(x^2), & O(x)
\end{bmatrix}.$$
Thus, $|\nabla G|$ is comparable to $1$ on $W$. By the coarea formula \cite{Federer:gmt}*{Theorem 3.2.11}, for each Borel function $\rho\colon W\to [0,\infty]$ we have
$$ \int_0^{C^{-1}} \int_{J_t\cap W} \rho \, d\mathcal H^1 dt=\int_0^{C^{-1}} \int_{G^{-1}(t)\cap W} \rho \, d\mathcal H^1 dt = \int_W \rho |\nabla G| \simeq  \int_W \rho.$$
This completes the proof. 
\end{proof}

The map $R$ is conjugate to $F(w) = w - 1 + cw^{-1} + O(w^{-2})$ near $\infty$ for some constant $c\in \C$. Let $\widetilde\Omega$ be the Fatou component $\Omega$ in this coordinate.
Consider the right half-plane $\{w \colon \Re(w) > C\}$ for some constant $C>0$ sufficiently large.

\begin{lemma}\label{lem:estimate}
    The following statements are true for a sufficiently large $C>0$.
    \begin{enumerate}[label=\normalfont(\arabic*)]
        \item\label{est:1} There exist two disjoint arcs $\widetilde{\gamma}^\pm$ in $ \partial\widetilde\Omega \cap \{w \colon \Re(w) > C \}$ each having one endpoint at $\infty$ and one endpoint at $\{w:\Re(w)=C\}$ that are backward invariant under $F$. 
        \item\label{est:2} There exist values $y^\pm \in \R$ so that $\widetilde{\gamma}^\pm\subset \{w: \Im(w)\in (y^-, y^+)\}$.
    \end{enumerate}
    For $t\in (0,C^{-1})$ let  $\widetilde U_t$ be the unbounded component of $\{w:\Re(w) > \frac{1}{t}\}\setminus \overline{\widetilde \Omega}$ and  $\widetilde {\mathbb I}_t$ be the interior of the linear set $\widetilde J_t\cap \partial \widetilde U_t$. For all $t\in (0,C^{-1})$ we have
    \begin{enumerate}[label=\normalfont(\arabic*)]\setcounter{enumi}{2}
        \item\label{est:3} $\widetilde U_t\cup {\widetilde {\mathbb I}_t}\subset \{w: \Im(w)\in (y^-,y^+)\}$ and $l(\widetilde {\mathbb I}_t) \simeq 1$. 
        \smallskip 
    \end{enumerate}
    Let  $\mathbb{I}_t = M^{-1}(\widetilde{\mathbb{I}}_t)$, $\gamma^\pm = M^{-1}(\widetilde{\gamma}^\pm)$, and $U_t = M^{-1}(\widetilde{U}_t)$. For $t\in (0,C^{-1})$ we have
    \begin{enumerate}[label=\normalfont(\arabic*)]\setcounter{enumi}{3}
        \item\label{eqn:I_t} $l(\mathbb{I}_t)  \simeq l(J_t\cap W)\simeq  t^2$ and
        \smallskip
        \item\label{eqn:U_t} $|U_t|  \simeq t^{3}.$
    \end{enumerate}
\end{lemma}

\begin{figure}
    \centering
    \begin{tikzpicture}
        \begin{scope}
        \draw (0,-1.5) node[below]{$\scriptstyle \Re(w)=C$}--(0,1.5);
        \draw[rounded corners, yshift=-0.3cm] (0,1)--(0.5,1.2)--(1,0.7)--(1.5,1)--(2,1)--(2.5,1.2)--(3,0.7)--(3.5,1) node[right ] {$\widetilde \gamma^+$};
        \draw[rounded corners,yshift=0.3cm] (0,-0.8)--(0.5,-0.5)--(1,-0.9)--(1.5,-1.2)--(2,-0.8)--(2.5,-0.5)--(3,-0.9)--(3.5,-1) node[right] {$\widetilde \gamma^-$}; 
        \draw[blue] (0,1)--(3.5,1) node[above,text=black] {$\Im(w)=y^+$};
        \draw[blue] (0,-1)--(3.5,-1) node[below,text=black] {$\Im(w)=y^-$};

        \draw (1.5,-1.5) node[below]{$\widetilde J_t$}--(1.5,1.5);
        \draw[red, line width =2pt] (1.5,-0.85)--(1.5,0.7) node[pos=0.5,left, text=black]{$\widetilde {\mathbb I}_t$};

        \node at (2.5,0.2) {$\widetilde U_t$};
    \end{scope}

    \draw[->] (4.5,0)--(6,0) node[pos=0.5,above] {$M^{-1}$};
        \begin{scope}[shift={(7,0)}]
            \draw[blue]   plot[smooth,domain=0:3.23] (\x, 0.15*\x*\x);
            \draw[blue]   plot[smooth,domain=0:3.23] (\x, -0.15*\x*\x);
            \fill[black] (0,0) circle (1pt) node[left] {$0$};
            \draw[black]   plot[smooth,domain=0:3.65] (\x, {0.08*\x*\x+0.03*\x*cos(1000*\x)});
            \draw[black]   plot[smooth,domain=0:3.65] (\x, {-0.08*\x*\x-0.03*\x*sin(1000*\x)});
            
            \draw (2,0) circle (2cm);
            \draw (1.3,0) circle (1.3cm);
            \draw [red,line width=2pt,domain=-22:23] plot ({1.3+1.3*cos(\x)}, {1.3*sin(\x)});

            \node[below] at (2,-2) {$V$};
            \node at (1.1,1.1) {$J_t$};
            \node at (2.7,1.4) {$\alpha^-$};
            \node at (2.7,-1.4) {$\alpha^+$};
            \node at (3.4, 1.1) {$\scriptstyle \gamma^-$};
            \node at (3.4, -1.1) {$\scriptstyle \gamma^+$};
            \node[right] at (2.6,0) {$\mathbb I_t$};
            \node at (2,0) {$U_t$};
        \end{scope}
    \end{tikzpicture}
    \caption{Illustration of the setup in Lemmas \ref{lem:estimate} and \ref{lem:compM}.}
    \label{fig:parabolic}
\end{figure}
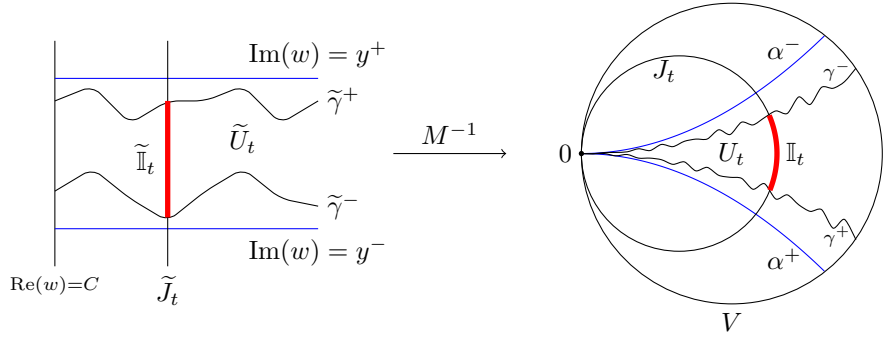

Here the length $l$ is measured by means of the Hausdorff $1$-measure, and $|U_t|$ is the area of $U_t$. See Figure \ref{fig:parabolic} for an illustration of the conclusions of the lemma.

\begin{proof}
Since the parabolic multiplicity of $R$ at $0$ is equal to $2$, for each $\delta>0$ there exists a neighborhood $N$ of $0$ such that $\partial\Omega\cap N$ is contained in the cone $\{re^{i\theta}\colon |\theta| < \delta\}$. Thus, if $C$ is sufficiently large, there exist two disjoint arcs $\widetilde{\gamma}^\pm$ in $ \partial\widetilde\Omega \cap \{w \colon \Re(w) > C\}$ with an endpoint at $\infty$ and an endpoint at $\{w:\Re (w)=C\}$ that are backward invariant under $F$.
This justifies Part \ref{est:1}.
We assume that $\widetilde{\gamma}^-$ lies below $\widetilde{\gamma}^+$, as in Figure \ref{fig:parabolic}.

    We consider the Fatou coordinate.
    Specifically, by \cite{Shi00}*{Proposition 2.2.1}, if $C$ is sufficiently large, there exists an injective holomorphic function $\Phi\colon \{w: \Re(w) >C\}\to \C$ so that $\Phi(F(w)) = \Phi(w) - 1$, and $\Phi$ has asymptotic expansion $w + O(\log w)$ as $w\to\infty$ {within a sector $\{x+iy: x>b-k|y|\}$}.
    Since the arcs $\Phi(\widetilde{\gamma}^\pm)$ are backward invariant under $z\mapsto z-1$, it is easy to check that they are contained in a horizontal strip and $l(\Phi(\widetilde{\mathbb{I}}_t))$ is uniformly bounded from below for $t\in (0,C^{-1})$. Note that $\Im(\log w)$ is bounded, so $\widetilde{\gamma}^\pm$ are contained in some horizontal strip $\{w: \Im(w)\in (y^-, y^+)\}$.
    This justifies Part \ref{est:2}, which also implies the first claim in \ref{est:3}.  Since on a horizontal strip we have $\Im(\log w) \to 0$ as $\Re(w) \to \infty$, we see that $l(\widetilde{\mathbb{I}}_t)$ is uniformly bounded from below for $t\in (0,C^{-1})$. This justifies Part \ref{est:3}.

    Part \ref{eqn:I_t} follows from \ref{est:3} and the observation that  $|(M^{-1})'(z)|\simeq t^2$ when $z\in \widetilde J_t\cap \{w: \Im(w)\in (y^-,y^+)\}$. For part \ref{eqn:U_t} note that $J_s\cap W \supset U_t\cap J_s \supset \mathbb I_s$ for $0<s<t<C^{-1}$. By \ref{eqn:I_t} we have $l(U_t\cap J_s) \simeq t^2$. Now, Lemma \ref{lem:compM}, applied to the characteristic function $\rho=\chi_{U_t}$, gives
    $$|U_t|= \int_W \chi_{U_t} \simeq \int_0^{t} \int_{J_s\cap W} \chi_{U_t}\, d\mathcal H^1 ds \simeq \int_0^t l(U_t\cap J_s) \, ds \simeq t^3.$$
    Thsi completes the proof. 
\end{proof}

\subsection{\texorpdfstring{$L^p$-distortion and $L^p$-quasidisks}{Lᵖ-distortion and Lᵖ-quasidisks}}
We first introduce some notation and basic properties of maps of finite $L^p$-distortion.
We refer the readers to \cite{IOZ21} for more details. A homeomorphism $f\colon \C\to \C$ in the Sobolev class $W^{1,1}_{\loc}(\C)$ is said to have \textit{finite distortion} if there exists a measurable function $K\colon \C\to [1,\infty)$ such that $|Df(z)|^2\leq K(z)J_f(z)$ for a.e.\ $z\in \C$. The smallest function $K(z)\geq 1$ with the above property is denoted by $K_f(z)$.
\begin{definition}Let $1\leq p\leq \infty$.
    A homeomorphism $f\colon \C\to \C$  is said to be a \textit{map of $L^p$-distortion} if $f$ is a map of finite distortion and $K_f \in L^p_{\loc}(\C)$. A domain $\mathbb{X} \subset \C$ is called an \textit{$L^p$-quasidisk} if there exists a homeomorphism $f\colon \C \to \C$ of $L^p$-distortion such that $f(\mathbb{X}) = \D$.
\end{definition}

Note that a homeomorphism $f\colon \C\to \C$ of class $W^{1,1}_{loc}(\C)$ is quasiconformal if and only if it is a map of $L^\infty$-distortion.
Moreover, a David homeomorphism is a map of $L^p$-distortion for all $1\leq p < \infty$, as a consequence of \eqref{exp_integral}.

Suppose $\mathbb{X} \subset \C$ is an $L^p$-quasidisk. Let $f\colon \C \to \C$ be a homeomorphism of $L^p$-distortion. We extend $f$ to $\widehat \C$ so that $f(\infty)=\infty$. Let $\Psi(z) = \frac{1}{\bar{z}}$ be the reflection along the unit circle. Then the map
\begin{align}\label{eqn:reflection}
g= f^{-1} \circ \Psi \circ f \colon \widehat \C\to \widehat \C
\end{align}
is a \textit{reflection} in the boundary $\partial \mathbb{X}$, i.e.,
    \begin{itemize}
        \item $g(\mathbb{X}) = \widehat\C \setminus \overline{\mathbb{X}}$, and
        \item $g(z) = z$ for $z \in \mathbb{X}$.
    \end{itemize}
The following result is a consequence of \cite{IOZ21}*{Theorem 3.2 and  (5.7)}.
\begin{lemma}\label{lemma:reflection}
    Let $\mathbb{X} \subset \C$ be an $L^p$-quasidisk, and let $g$ be the reflection in $\partial \mathbb{X}$ given by \eqref{eqn:reflection}.
    Then for any bounded domain $U \subset \C$ so that $f^{-1}(0) \notin \overline{U}$, we have $g \in W^{1,1}(U)$ and
    \begin{align*}
    \left(\int_U \frac{\left|Dg\right|^p}{\left|J_g\right|^{\frac{p-1}{2}}}\right)^{\frac{1}{p}} \left(\int_U \left|J_g\right|^{\frac{1}{2}}\right)^{\frac{p-1}{p}}  &\leq \| K_f \|^{\frac{1}{2}}_{L^p(U)}\| K_f \|^{\frac{1}{2}}_{L^p(g(U))} \left|U\right|^{\frac{p-1}{2p}}\left|g(U)\right|^{\frac{p-1}{2p}}.
    \end{align*}
\end{lemma}
Here, integration is with respect to Lebesgue measure in the plane; also, in the end of the line, $|\cdot |$ denotes the area of a planar set.

\begin{lemma}\label{lem:subseq}
    Let $\{a_n\}_{n\in \N}$ be a sequence of positive real numbers and $r>0$.  Then  one of the following alternatives holds.
    \begin{enumerate}[label=\normalfont(\arabic*)]
        \item $a_{n}\leq r^{n}$ for all but finitely many $n\in \N$.
        \item There exists a subsequence $\{a_{k_n}\}_{n\in \N}$ of $\{a_n\}_{n\in \N}$ and a constant $M>0$ such that $r^{k_n}< a_{k_n}\leq M a_{k_n+1}$ for each $n\in \N$.
    \end{enumerate}
\end{lemma}
\begin{proof}
    Suppose $a_n>r^{n}$ for infinitely many $n\in \N$. By passing to a subsequence, we assume that this is the case for all $n\in \N$. If $\lim_{n\to\infty}a_n/a_{n+1}=\infty$, then $\lim_{n\to\infty}a_{n+1}/a_n=0$, so $a_{n}\leq r^{n}$ for all sufficiently large $n\in \N$. This is  a contradiction. Therefore, $\liminf_{n\to\infty}a_n/a_{n+1}<\infty$ and there exists a constant $M>0$ and a subsequence $\{a_{k_n}\}_{n\in \N}$ of $\{a_n\}_{n\in \N}$ such that $a_{k_n}\leq Ma_{k_n+1}$.
\end{proof}

We are ready to prove the following more precise statement which immediately implies Theorem \ref{theorem:basins} \ref{theorem:basins:no_david}.

\begin{theorem}
    Let $R$ be a rational map, $a$ be a parabolic fixed point with $R'(a)=1$ that has parabolic multiplicity $2$, and $\Omega$ be an immediate basin of $a$. Then $\Omega$ is not an $L^p$-quasidisk for $p \geq 5$.
\end{theorem}

In the proof we use the notation introduced above. 

\begin{proof}
    Suppose, for the sake of contradiction, that $\Omega$ is an $L^p$-quasidisk for some $p \geq 5$. Let $\epsilon \in (0,C^{-1})$ and $t\in (0,\epsilon)$. Consider the reflection $g$ defined above. The closure of the set $\mathbb I_t\cup g(\mathbb I_t)$ contains a Jordan curve that bounds the set $U_t\cup g(U_t)$ and the point $0$. We have $\dist (0, \mathbb{I}_t) \simeq t$ because $\widetilde{\mathbb I}_t\subset \{w: \Re(w)=1/t, \, \Im(w)\in (y^-,y^+)\}$. Thus 
    $$l(\mathbb I_t)+ l(g(\mathbb I_t)) \geq l(\mathbb I_t \cup g(\mathbb I_t)) \geq  \diam (\mathbb I_t\cup g(\mathbb I_t))\gtrsim t.$$
    By Lemma \ref{lem:estimate} \ref{eqn:I_t}, we have $l(\mathbb{I}_t) \simeq t^2$, so  $l(g(\mathbb{I}_t)) \geq l (\mathbb{I}_t)$ and $l(g(\mathbb I_t)) \gtrsim t$ for all sufficiently small $t>0$. 
    This implies  that $t \lesssim l(g(\mathbb{I}_{t})) \lesssim l(g(U_\epsilon\cap J_t))$, given that $\mathbb I_t\subset U_\epsilon\cap J_t$. Therefore,  for all sufficiently small $\epsilon>0$ and for $t\in (0,\epsilon)$, we have
    \begin{align}\label{theorem:lp_qdisk:l1}
     \epsilon^2\lesssim \int_0^{\epsilon} l(g(U_\epsilon\cap J_t))\, dt.   
    \end{align}
    Next, by the isoperimetric inequality, $|g(U_t)|^{\frac{1}{2}} \lesssim l(g(\mathbb{I}_t))\lesssim l(g(U_\epsilon\cap J_t))$. Integrating this from $\epsilon/2$ to $\epsilon$, we obtain  
    \begin{align}\label{theorem:lp_qdisk:l2}
        (\epsilon - \epsilon/2)  |g(U_{\epsilon/2})|^{\frac{1}{2}} \lesssim \int_{\epsilon/2}^{\epsilon} l(g(U_\epsilon\cap J_t))\, dt \lesssim \int_{0}^{\epsilon} l(g(U_\epsilon\cap J_t))\, dt.
    \end{align}
    Combining \eqref{theorem:lp_qdisk:l1} and \eqref{theorem:lp_qdisk:l2}, we obtain
    \begin{align}\label{theorem:lp_qdisk:l3}
        \max\{\epsilon^2, \epsilon |g(U_{\epsilon/2})|^{1/2} \}\lesssim \int_0^\epsilon l(g(U_\epsilon\cap J_t))\, dt.
    \end{align}
    By H\"older's inequality (see \cite{IOZ21}*{Lemma 5.2}), we have that the following length estimate for a.e.\ $t\in (0,C^{-1})$:
\begin{align}\label{theorem:lp_qdisk:l4}
    l(g(U_\epsilon\cap J_t)) &\leq \int_{U_\epsilon\cap J_t} \left| Dg \right|\, d\mathcal H^1 = \int_{U_\epsilon\cap J_t} \frac{\left|Dg\right|}{\left|J_g\right|^{\frac{p-1}{2p}}} \left|J_g\right|^{\frac{p-1}{2p}}\,d\mathcal H^1 \\
    &\leq \left(\int_{U_\epsilon\cap J_t} \frac{\left|Dg\right|^p}{\left|J_g\right|^{\frac{p-1}{2}}}\,d\mathcal H^1\right)^{\frac{1}{p}} \left(\int_{U_\epsilon\cap J_t} \left|J_g\right|^{\frac{1}{2}}\,d\mathcal H^1\right)^{\frac{p-1}{p}}. \nonumber
\end{align}
    Combining \eqref{theorem:lp_qdisk:l3}, \eqref{theorem:lp_qdisk:l4}, and H\"older's inequality, we obtain 
    \begin{align*}
        \max\{\epsilon^2, \epsilon & |g(U_{\epsilon/2})|^{1/2} \}  \lesssim \int_0^{\epsilon} l(g(U_\epsilon\cap J_t))\, dt \\
        &\lesssim \left(\int_0^{\epsilon} \int_{U_\epsilon\cap J_t} \frac{\left|Dg\right|^p}{\left|J_g\right|^{\frac{p-1}{2}}}\,d\mathcal H^1 \, dt \right)^{\frac{1}{p}} \left(\int_0^{\epsilon}\int_{U_\epsilon\cap J_t} \left|J_g\right|^{\frac{1}{2}}\,d\mathcal H^1\, dt\right)^{\frac{p-1}{p}}.
    \end{align*}
    By Lemma \ref{lem:compM}, we have
    \begin{align*}
        \max\{\epsilon^2, \epsilon |g(U_{\epsilon/2})|^{1/2} \} &\lesssim \left(\int_{U_{\epsilon}} \frac{\left|Dg\right|^p}{\left|J_g\right|^{\frac{p-1}{2}}}\right)^{\frac{1}{p}} \left(\int_{U_{\epsilon}} \left|J_g\right|^{\frac{1}{2}}\right)^{\frac{p-1}{p}}.
    \end{align*}   
    Therefore, by the area estimate in Lemma \ref{lemma:reflection} we have
    \begin{align*}
         \max\{\epsilon^2, \epsilon |g(U_{\epsilon/2})|^{1/2} \} &\lesssim \| K_f \|^{\frac{1}{2}}_{L^p(U_{\epsilon})}\| K_f \|^{\frac{1}{2}}_{L^p(g(U_{\epsilon}))} \left|U_{\epsilon}\right|^{\frac{p-1}{2p}}\left|g(U_{\epsilon})\right|^{\frac{p-1}{2p}}.
    \end{align*}
    Combined with Lemma \ref{lem:estimate} \ref{eqn:U_t}, this gives
    \begin{align}\label{theorem:lp_qdisk:l5}
         \max\{\epsilon^{2p},  |g(U_{\epsilon/2})|^{p} \}&\lesssim C(\epsilon) \epsilon^{p-3}\left|g(U_{\epsilon})\right|^{p-1},
    \end{align}
    where $C(\epsilon)= \| K_f \|^{p}_{L^p(U_{\epsilon})}\| K_f \|^{p}_{L^p(g(U_{\epsilon}))}$, which satisfies $\lim_{\epsilon\to 0}C(\epsilon)=0$.

    We apply \eqref{theorem:lp_qdisk:l5} for $\epsilon_n=2^{-n}$, where $n\in \N$, and we set $a_n=|g(U_{2^{-n}})|$. Then, for each $n\in \N$, we have
    \begin{align}\label{theorem:lp_qdisk:l6}
        \max\{ \epsilon_n^{2p},   a_{n+1}^{p}\} \lesssim C(\epsilon_n)  \epsilon_n^{p-3} a_n ^{p-1}.
    \end{align}
    We apply Lemma \ref{lem:subseq} with $r=1/4$. If the first alternative of the lemma holds, then $a_n\leq \epsilon_n^2$ for all but finitely many $n\in \N$. Therefore, by \eqref{theorem:lp_qdisk:l6}, for all large $n\in \N$ we have
    $$\epsilon_n^{2p}\lesssim C(\epsilon_n) \epsilon_n^{p-3} \epsilon_n^{2p-2}.$$
    Since $p\geq 5$, this leads to a contradiction. Thus, the second alternative of Lemma \ref{lem:subseq} holds and there exists a subsequence $\{a_{k_n}\}_{n\in \N}$ of $\{a_n\}_{n\in \N}$ such that $\epsilon_{k_n}^2\lesssim a_{k_n}\lesssim a_{k_n+1}$ for all $n\in \N$. By \eqref{theorem:lp_qdisk:l6}, we conclude that
    $$ a_{k_n}^{p} \lesssim  a_{k_n+1}^{p} \lesssim C(\epsilon_{k_n}) \epsilon_{k_n}^{p-3} a_{k_n}^{p-1}$$
    so 
    $$1\lesssim C(\epsilon_{k_n}) \epsilon_{k_n}^{p-3} a_{k_n}^{-1}\lesssim C(\epsilon_{k_n}) \epsilon_{k_n}^{p-3} \epsilon_{k_n}^{-2}$$
    for all $n\in \N$.  Since $p\geq 5$, this leads again to a contradiction.
\end{proof}

\section{Uniformization of parabolic basins}\label{section:qdF}

In this section we prove parts \ref{theorem:basins:david} and \ref{theorem:basins:qc} of Theorem \ref{theorem:basins}. We first prove the following preliminary statement.
\begin{lemma}\label{lem:expand}
     Let $R$ be a rational map, $a$ be a fixed point with $R'(a) = 1$ that has parabolic multiplicity $\nu\geq 2$, and $\Omega$ be an immediate basin of $a$. Suppose that $\Omega$ is a Jordan region and the critical and post-critical sets intersect $\partial \Omega$ only at the point $a$. For $\epsilon>0$ let $N_\epsilon$ be the $\epsilon$-neighborhood of $\gamma\coloneqq\partial \Omega \setminus  \{a\}$ in the hyperbolic metric $d_X$ of $X\coloneqq\widehat\C \setminus  P_R$.
     \begin{enumerate}[label=\normalfont(\arabic*)]
         \item\label{lem:expand:1} For all sufficiently small $\epsilon>0$ and for each component $I$ of $R^{-1}(\gamma) \cap \gamma$, the component $N_{\epsilon, I}$ of $R^{-1}(N_\epsilon)$ containing $I$ is a subset of $N_\epsilon$.
         \item\label{lem:expand:2} Suppose that $\nu\geq 3$. Then for all sufficiently small $\epsilon>0$, $N_\epsilon$ is contained in a simply connected domain in $\widehat\C \setminus \mathcal{V}$, where $\mathcal{V}$ is the set of critical values of $R$. In particular, the map $R\colon N_{\epsilon, I} \to N_\epsilon$ is a conformal isomorphism.
     \end{enumerate}
\end{lemma}
\begin{proof}
Let $P_R$ be the post-critical set of $R$.
To prove \ref{lem:expand:1}, by choosing an appropriate normalization, we may assume that $X \subset \C$, i.e., $\infty \in P_R$.
    We set $Y=R^{-1}(X)$ and note that $R\colon Y \to X$ is a covering map.

    Since $R$ has no critical point on $\partial \Omega$, if $\gamma'$ is a component of $R^{-1}(\gamma)$ that is disjoint from $\gamma$, then $\overline{\gamma'}$ is disjoint from $\partial \Omega$. Thus, we can choose $\epsilon>0$ small enough so that for any component $\gamma'$ of $R^{-1}(\gamma)$ that is disjoint from $\gamma$, we have
    \begin{align}\label{lem:expand:3e}
        d_Y(\gamma\cap Y, \gamma') \geq  3\epsilon.
    \end{align}

    Let $U$ be the $\epsilon$-neighborhood of $R^{-1}(\gamma)$ in the hyperbolic metric of $Y$. Note that $U$ may be disconnected as $R^{-1}(\gamma)$ may be disconnected.
    Since $R\colon Y \to X$ is a local isometry, {we have $U=R^{-1}(N_\epsilon)$,} so  $N_{\epsilon, I}$ is the component of $U$ that contains $I$. Let $\gamma'$ be a component of $R^{-1}(\gamma)$ that is disjoint from $\gamma$. By \eqref{lem:expand:3e}, the $\epsilon$-neighborhood of $\gamma'$ in $Y$ is disjoint from the $\epsilon$-neighborhood of $R^{-1}(\gamma) \cap \gamma$ in $Y$. Thus, we have $N_{\epsilon, I} \cap R^{-1}(\gamma) = N_{\epsilon, I} \cap \gamma$. 
    
    Let $x \in N_{\epsilon, I}$, so $R(x)\in N_\epsilon$.  Let $\alpha$ be a geodesic in $X$ that connects $R(x)$ and $\gamma$ so that the hyperbolic length $l_X(\alpha)$ is equal to $d_X(R(x), \gamma)<\epsilon$ and $\alpha$ lies in $N_\epsilon$. 
    Let $\widetilde{\alpha}$ be the lift of $\alpha$ starting at $x$. Then $\widetilde \alpha$ lies  in a component of $R^{-1}(N_\epsilon)$, so it must lie in $N_{\epsilon, I}$. Moreover,  $\widetilde \alpha$ a geodesic in $Y$ connecting $x$ and $R^{-1}(\gamma)$ and $l_Y(\widetilde{\alpha}) = l_X(\alpha)$. Since $N_{\epsilon, I} \cap R^{-1}(\gamma) = N_{\epsilon, I} \cap \gamma$, $\widetilde \alpha$ connects $x$ and $R^{-1}(\gamma) \cap \gamma$.  Thus, $d_{X}(x, \gamma)\leq l_X(\widetilde \alpha)$. Note that $Y \subset X$, so by the Schwarz--Pick lemma \cite{BeardonMinda}*{Theorem 10.5} the inclusion map decreases distances. Thus, $l_X(\widetilde{\alpha}) \leq l_Y(\widetilde{\alpha})$ (c.f.\ the argument in \cite{McM94}*{Theorem 3.5}). Therefore, we have 
    $$
        d_{X}(x, \gamma)\leq l_X(\widetilde \alpha)\leq l_Y(\widetilde \alpha)=l_X(\alpha)= d_{X}(R(x), \gamma) < \epsilon.
    $$
    Thus, $x\in N_\epsilon$. This proves part \ref{lem:expand:1}.
    
    To prove \ref{lem:expand:2}, we note that by the theory of parabolic points there exist $2(\nu-1)$ vectors at $a$ corresponding to the attracting and repelling directions. For a small $\delta>0$ we consider, for each attracting and repelling direction $v$, an open cone centered at $a$ of radius $1$ and total angle $\delta$ that is symmetric with respect to the vector $v$. Then there exists a small $r>0$ such that the union $D_\delta$ of these cones contains the part of the Julia set that lies inside $B(a,r)$, except for the point $a$, and also contains the tail of the orbit of each critical point that is attracted to $a$; see Figure \ref{fig:parabolic_cones}. In particular, the post-critical set in $B(a,r)$ lies in $D_\delta\cup \{a\}$.  By choosing $r$ sufficiently small, we may assume that $B(a, r)$ contains no critical value of $R$, except possibly for $a$. 

    \begin{figure}
        \centering
        \begin{tikzpicture}
            \begin{scope}
            \clip (-2.7,-2.6) rectangle (3,2.5);
            \node at (0.3,-0.3) {\includegraphics[width=0.5\linewidth]{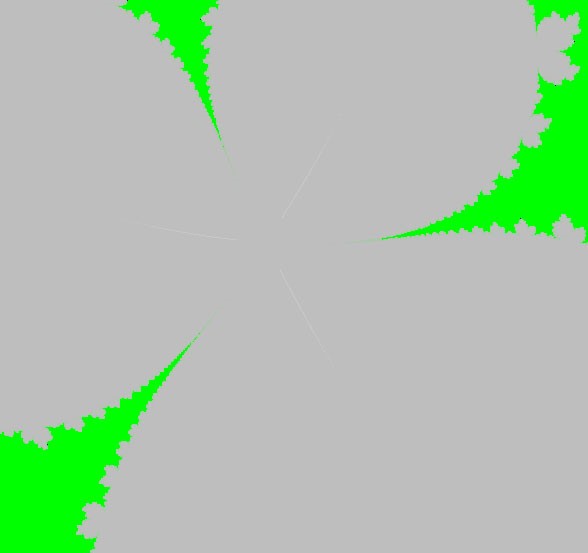}};
            \fill (0,0) circle (1.5pt) node[shift={(-0.3,0.13)}] {$a$};
            \node at (1.3,1.1) {$\Omega$};
            
            \draw[pattern=north east lines] (0,0) --(3.8,0.7)--(3.8,0)--cycle;
            \draw[pattern=north east lines, pattern color=red] (0,0)-- (2.3,2.7)--(1.3,2.7)--cycle;
            \draw[pattern=north east lines]  (0,0)--(-0.7,2.7)--(-1.7,2.7)--cycle;
            \draw[pattern=north east lines, pattern color=red] (0,0) --(-3.8,-0.7)--(-3.8,0)--cycle;
            \draw[pattern=north east lines] (0,0)-- (-2.3,-2.7)--(-1.3,-2.7)--cycle;
            \draw[pattern=north east lines,pattern color=red]  (0,0)--(0.7,-2.7)--(1.7,-2.7)--cycle;

            \draw[pattern=north east lines,pattern color=blue] (0,0)--(3.05,1.3)--(3, 1.8)--cycle;
            \draw[pattern=north east lines,pattern color=blue] (0,0)--(-3,2.8)--(-2.5, 3)--cycle;
            \end{scope}
            
            \node at (3.3,1.6) {$E^+$};
            \node at (-2.1,2.7) {$E^-$};
            
        \begin{scope}[shift={(6,0)}]
            \clip (-2,-2.6) rectangle (3,2.5);
            \node at (0.3,-0.3) {\includegraphics[width=0.5\linewidth]{parabolic_cones.jpg}};
            \fill (0,0) circle (1.5pt) node[anchor=north east] {$a$};
            
            \fill[rounded corners=10, pattern=north east lines,pattern color=purple, opacity=0.3] (-0.8,1)--(-1,0)--(1,-1)--(2,-0.5)--(2.5,0.3)--(3,1)--(3,3)--(-0.6,3)--(-1,2)--(-0.8,1);
            \draw[rounded corners=10] (-0.8,1)--(-1,0)--(1,-1)--(2,-0.5)node[anchor=south east,opacity=1]{$W$}--(2.5,0.3)--(3,1)--(3,3)--(-0.6,3)--(-1,2)--(-0.8,1);

            \definecolor{mygray}{RGB}{190,190,190}
            \draw[fill=mygray] (1,1.4) circle (0.9cm);

            \fill (-0.8,1) circle (1.5pt);
            \fill (1,0.5) circle (1.5pt);
            \draw (0,0)--(1,0.5) node[below] {$L^+$};
            \draw (0,0)--(-0.8,1) node[left] {$L^-$};
        \end{scope}
        \end{tikzpicture}
        \caption{Left: An immediate basin $\Omega$ of a parabolic point $a$ with multiplicity $\nu=4$. Shown are the six cones $D_\delta$ (with black and red stripes) containing $P_R\setminus \{a\}$ and the two cones $E^\pm$ containing the segments $L^\pm$. Right: The segments $L^\pm$ and the simply connected region $W$ that contains $N_\epsilon$.}
        \label{fig:parabolic_cones}
    \end{figure}

    Consider a line segment $L^+$ (resp.\ $L^-$) of length $r/2$ in $\Omega\cap X$ (resp.\ in $X\setminus \overline \Omega$) with an endpoint at $a$ so that $L^+$ (resp.\ $L^-$) is disjoint from $D_\delta$; see Figure \ref{fig:parabolic_cones}. The existence of $L^-$ is precisely what requires that the multiplicity $\nu$ satisfies $\nu\geq 3$.  Consider cones $E^{\pm}$ centered at $a$, of radius $1$ and angle $\delta$, in the direction of $L^\pm$ and symmetric with respect to $L^\pm$. If $\delta$ is small, we can ensure that the cones $E^\pm$ are disjoint from $D_\delta$. 
    For $z\in X$, let $\delta_X(z)$ be the Euclidean distance between $z$ and $\partial X = P_R$. Then $\delta_X(z)$ is proportional to $|z-a|$ for $z \in E^{\pm}\cap B(a,r)$.
    Let $x_z \in \partial X$ be a point so that $|z-x_z| = \delta_X(z)$.
    By the parabolic dynamics near $a$, we can find a point $y_z \in \partial X$ so that $|y_z - x_z|$ is proportional to $|z-x_z|$ for $z \in E^\pm \cap B(a,r)$ (in fact, for any point $y$ that is attracted to the parabolic point $a$ we have $|a-R^k(y)| \simeq k^{-\nu+1}$ for all $k\in \N$).
    Therefore, the quantity
    $$
    \beta_X (z) \coloneqq \inf\left\{\left|\log \frac{|z-u|}{|v-u|}\right|\colon u, v\in \partial X, |z-u| = \delta_X(z)\right\}
    $$
    is uniformly bounded from above in $E^\pm \cap B(a,r)$. By \cite{BeardonPommerenke}*{Theorem 1}, there exists $M>0$ such that the hyperbolic density $\rho_X$ satisfies $\rho_X(z) \geq M^{-1}\delta_X(z)^{-1}$ for $z\in E^\pm \cap B(a,r)$. This implies that the hyperbolic distance in $X$ between $\gamma$ and $L^\pm$ is uniformly bounded from below. As a consequence, we have $N_\epsilon  \subset X\setminus (L^+\cup L^-)$ provided that $\epsilon$ is small enough.  We consider a Jordan curve that surrounds $\overline \Omega\cup L^-$ and passes through the endpoint of $L^-$ that is different from $a$. We also consider a Jordan curve inside $\Omega\setminus L^+$ that passes through the endpoint of $L^+$ that is different from $a$. Let $W$ be the simply connected region that is bounded by these two Jordan curves and $L^+\cup L^-\cup \{a\}$; see Figure \ref{fig:parabolic_cones}. Then we may ensure (by choosing the above Jordan curves very close to $\partial \Omega$) that $W$ does not contain any critical value of $R$; this uses the assumption that $P_R\cap \partial \Omega=\{a\}$. Finally, if we choose $\epsilon$ to be small enough, we can ensure that $N_\epsilon \subset W$. This proves part \ref{lem:expand:2}.    
\end{proof}

\begin{proof}[Proof of Theorem \ref{theorem:basins}: \ref{theorem:basins:david} and \ref{theorem:basins:qc}]
Note that $\Omega$ is invariant under $R$. Let $\phi\colon  \D \to \Omega$ be a conformal map, so that its continuous extension to the boundaries satisfies $\phi(1) = a$.
Then $f\coloneqq\phi^{-1} \circ R \circ \phi$ is a Blaschke product with a parabolic fixed point at $1$. We denote its degree by $r+1$. Similarly, let $\Omega^{*}\coloneqq \widehat\C \setminus \overline{\Omega}$. Let $\psi\colon  \widehat \C\setminus \overline \D \to \Omega^{*}$ be a conformal map so that $\psi(1) = a$. Since $R$ leaves $\Omega$ invariant and has no critical points on $\partial \Omega$, for every point $x\in \partial \Omega$, there exists a neighborhood $U_x$ so that $R(U_x \cap \Omega^{*}) \subset \Omega^{*}$. Thus $g\coloneqq\psi^{-1} \circ R \circ \psi$ is defined in a neighborhood of $\mathbb S^1$ in $\widehat \C \setminus \overline \D$. By the Schwarz reflection principle, $g$ defines an analytic covering map on $\mathbb S^1$ of degree $r+1$. Note that the preimages $f^{-1}(1) = \{a_0 = 1,a_1,\dots , a_r\}$ and $g^{-1}(1) = \{b_0 = 1,b_1,\dots , b_r\}$, written in cyclic order, define Markov partitions for $f$ and $g$, respectively. The map $h=\psi^{-1}\circ \phi$ on $\mathbb S^1$ is a topological conjugacy between $f$ and $g$ that maps $a_k$ to $b_k$, $k\in \{0,\dots,r\}$.

\begin{lemma}
    If $a$ has parabolic multiplicity $\nu\geq 3$, then $g$ has a parabolic fixed point at $1$.    
    If $\nu=2$, then $g$ has a repelling fixed point at $1$.
\end{lemma}
\begin{proof}
    Suppose that $\nu\geq3$.
    Since $g$ maps the circle  $\mathbb S^1$ to itself and $g$ is analytic at $1$, we see that $g'(1)$ is a real number. Suppose that $|g'(1)|\neq 1$, so $1$ is either a repelling or an attracting fixed point for $g$. In the first case, there exists some neighborhood $U$ of $a$ so that $U \cap \Omega^{*} \subset R(U \cap \Omega^{*})$,  while in the second case, there exists some neighborhood $V$ of $a$ so that $R(V \cap \Omega^{*}) \subset V \cap \Omega^{*}$. Either is not possible, as $a$ has both attracting and repelling petals in $\Omega^{*}$. Thus, $g$ has a parabolic fixed point at $1$. 
    The proof for the other case is similar.
\end{proof}

We verify that $f$ satisfies conditions \ref{condition:hp}, \ref{condition:uv} and \ref{condition:qs}. 
As remarked, $f$ is a Blaschke product with a parabolic fixed point at $1$ and $f|_{\overline \D}$ is conjugate to $R|_{\overline \Omega}$. Points in the unit disk and its exterior lie in the Fatou set of $f$ and are attracted to $1$. Thus, the Julia set is contained in $\mathbb S^1$. By the conjugacy, we see that the Julia set of $f$ is all of $\mathbb S^1$. This implies that the multiplicity of the parabolic point $1$ of $f$ is equal to $3$. Moreover, $f|_{\mathbb S^1}$ is expansive; see the discussion in \cite{LyubichMerenkovMukherjeeNtalampekos:David}*{Example 3.5}. We apply Lemma \ref{lem:expand} to $f$ to find a neighborhood $Y$ of $\mathbb S^1\setminus \{1\}$ so that $f^{-1}(Y)\subset Y$. For $k\in \{0,\dots,r\}$ let $X_k$ be the component of $f^{-1}(Y)$ that contains the arc $\arc{(a_k, a_{k+1})}$. Then by Lemma \ref{lem:expand}, $X_k \subset Y$ and $f\colon X_k \to Y$ is a conformal map.
Therefore, $\mathcal{P}(f; \{a_0,\dots, a_r\})$ satisfies condition \ref{condition:uv}.
Conditions \ref{condition:hp} and \ref{condition:qs} are satisfied as $f$ is analytic and expansive; see Remark \ref{remark:hyp_par_analytic}. Since $f$ is conjugate to $g$ on $\mathbb S^1$, we conclude that $g$ is also expansive.

In order to prove part \ref{theorem:basins:qc} of the theorem, by the fundamental theorem of conformal welding \cite{AstalaIwaniecMartin:quasiconformal}*{Theorem 5.10.1}, it suffices to show that the conjugacy $h$ is a quasisymmetry. 
We choose $\epsilon$ small enough so that Lemma \ref{lem:expand} applies to the map $R$ and let $\widetilde{V}\coloneqq \psi^{-1}(N_\epsilon\cap \Omega^{*})$.
Let $V\coloneqq \widetilde{V}\cup (\mathbb S^1\setminus \{1\}) \cup \sigma(\widetilde{V})$, where $\sigma$ is the reflection along the circle $\mathbb S^1$. For $k\in \{0,\dots,r\}$ let $U_k$ be the component of $g^{-1}(V)$ that contains the arc $\arc{(b_k, b_{k+1})}$. Then by Lemma \ref{lem:expand}, $U_k \subset V$ and $g\colon U_k \to V$ is a conformal map.
Therefore, $\mathcal{P}(g; \{b_0,\dots , b_r\})$ satisfies condition \ref{condition:uv}.
Conditions \ref{condition:hp} and \ref{condition:qs} are satisfied as $g$ is analytic and expansive on $\mathbb S^1$.  Finally, note that the conjugacy $h$ maps the parabolic point $1$ of $f$ to the parabolic point $1$ of $g$. By Theorem \ref{theorem:extension_generalization}, the map $h$ is quasisymmetric. This completes the proof of \ref{theorem:basins:qc}.

Next, we prove part \ref{theorem:basins:david}. Since $R$ is locally univalent on $\partial \Omega$, we conclude that $g$ is locally univalent on $\mathbb S^1$, so it has no critical points on $\mathbb S^1$. For a small $\delta>0$ consider two circles $C^{\pm}$ centered at $0$ and of radii $1\pm \delta$, so that the annular region $A_\delta$ between $C^{\pm}$ lies in the range of $g$. 
Since $g$ has no critical points on $\mathbb  S^1$, we may choose a smaller $\delta>0$ so that  the component $A_\delta'$ of $g^{-1}(A_\delta)$ that contains $\mathbb S^1$ has no critical points. Thus, $g\colon  A_\delta' \to A_\delta$ is a covering map.
We use the local dynamics at the repelling point $1$ of $g$ and choose a smaller $\delta$ if necessary to find a Jordan arc $\mathcal L$ passing through the point $1$, so that $\mathcal{L} \subset g(\mathcal{L})$ and the endpoints of $\mathcal L$ are on $C^{\pm}$, but otherwise $\mathcal L$ is disjoint from $C^{\pm}$. Then $C^{\pm}$ and $\mathcal L$ bound a simply connected region $Z$. By applying Lemma \ref{lem:expand} \ref{lem:expand:1} to $R$, upon choosing a small $\epsilon$, we can find a neighborhood $V'$ of $\mathbb S^1\setminus \{1\}$ so that $g^{-1}(V') \subset V'$ and $V'\setminus \mathcal L\subset Z$. We define $V\coloneqq V'\setminus \mathcal{L}$. For $k\in \{0,\dots,r\}$ let $U_k$ be the component of $g^{-1}(V)$ that contains the arc $\arc{(b_k, b_{k+1})}$. Then by construction, $U_k \subset V$, and $g\colon U_k \to V$ is a conformal map, because $Z$ is simply connected. So $\mathcal{P}(g; \{b_0,\dots, b_r\})$ satisfies conditions \ref{condition:hp}, \ref{condition:uv} and \ref{condition:qs}. Note that the conjugacy $h^{-1}$ maps the hyperbolic point $1$ of $g$ to the parabolic point $1$ of $f$. By Theorem \ref{theorem:extension_generalization}, the map $h^{-1}$ extends to a David map on the disk, which we denote by $h^{-1}$. 

Define a Beltrami coefficient $\mu$ as follows. In $\D$ we let $\mu$ be the pullback of the standard complex structure under $h^{-1}$. In $\widehat \C\setminus \overline \D$ we define $\mu$ to be the standard complex structure. By the David integrability theorem (Theorem \ref{theorem:integrability_david}) there exists a David homeomorphism $H$ of $\widehat \C$ with $\mu_H=\mu$. Consider the map
$$\alpha = \begin{cases} \phi \circ h^{-1}\circ H^{-1} & \textrm{in $H(\D)$}\\
\psi\circ H^{-1} & \textrm{in $H(\widehat \C\setminus \D)$.}
\end{cases}$$
The two definitions agree on $J=H(\mathbb S^1)$ because $h=\psi^{-1}\circ \phi$, so $\alpha$ is a homeomorphism of $\widehat \C$.  By Theorem \ref{theorem:stoilow}, $h^{-1}\circ H^{-1}$ is a conformal map, so $\alpha$ maps conformally $H(\D)$ onto $\Omega$. Also, $H^{-1}$ is conformal in $H(\widehat \C\setminus \overline \D)$ so $\alpha$ maps conformally $H(\widehat \C\setminus \overline \D)$ onto $\Omega^*$. By Theorem \ref{theorem:removable}, $\alpha$ is conformal on $\widehat \C$ and thus it is a M\"obius transformation. The map $\alpha \circ H$ is a David homeomorphism (e.g.\ by Proposition \ref{prop:david_qc_invariance}) that maps the unit disk onto $\Omega$. This proves part \ref{theorem:basins:david}.
\end{proof}

\section{Blaschke products}\label{section:blaschke}

\begin{proof}[Proof of Theorem \ref{theorem:blaschke_conditions}]
We will verify condition \ref{condition:uv}. Assuming that, conditions \ref{condition:hp}, \ref{condition:qs} and \ref{condition:qs_strong} are trivially satisfied; see Remark \ref{remark:hyp_par_analytic}. To verify the condition \ref{condition:uv}, we consider the hyperbolic and parabolic cases separately.

\subsection*{Hyperbolic case}

Choose $r_0 < 1$, and let $C_{r_0}$ be the circle centered at the origin with radius $r_0$. Let $A_{r_0}$ be the annulus bounded by $C_{r_0}$ and $\mathbb S^1$. Since $f$ is hyperbolic, $f$ is expanding near $\mathbb S^1$. Thus, for $r_0$ sufficiently close to $1$, we have $A'_{r_0}\coloneqq f^{-1}(A_{r_0}) \subset A_{r_0}$. Since $a_k$ is either a repelling periodic point or is eventually mapped to some repelling periodic point, we can inductively construct proper arcs $\mathcal{L}_k \subseteq A'_{r_0}$ connecting $a_k$ and $\partial A'_{r_0}$ so that $\mathcal{L}_j \subset f(\mathcal{L}_k)$ if $f(a_k) = a_j$.
These proper arcs cut $A'_{r_0}$ into $r+1$ regions $\widetilde{U_k}$, where $\partial \widetilde{U_k} \cap \mathbb S^1 = \arc{[a_k,a_{k+1}]}$, $k=0,\dots,r$.
Let $U_k\coloneqq \widetilde{U_k} \cup \inter(A_k) \cup r(\widetilde{U_k})$ where $r$ is the reflection along $\mathbb{S}^1$, and $V_k\coloneqq  f(U_k)$.
Then it is easy to check that condition \ref{condition:uv} is satisfied in this case.

\subsection*{Parabolic case}
Note that $a_0$ has parabolic multiplicity $3$; see the discussion in \cite{LyubichMerenkovMukherjeeNtalampekos:David}*{Example 3.5}. By Lemma \ref{lem:expand}, there exists a neighborhood $V'$ of $\mathbb{S}^1\setminus \{a_0\}$ so that for each component $U'$ of $f^{-1}(V')$, we have $f: U' \to V'$ is conformal and $U' \subset V'$. Note that $a_0 \in \partial V'$, as otherwise $V'$ contains $\mathbb S^1$, and the map $f\colon U' \to V'$ has degree at least $2$.
Let $l$ be the smallest number so that for any $a_k \in \{a_0,\dots, a_r\}$ in the grand orbit of the parabolic fixed point $a_0$, we have $f^l(a_k) = a_0$.
Let $U'_0,\dots, U'_m$ be the components of $f^{-l}(V')$ in cyclic order.
Observe that $f\colon U_i' \to f(U_i')$ is conformal and $f(U_i')$ contains $U_j'$ whenever $U_j' \cap f(U_i') \neq \emptyset$.

We take a small round disk $D_{a_0}$ centered at $a_0$ of radius $\delta$. Let $a \in f^{-l}(a_0)$ with pre-period $q$. Denote by $D_a$ the component of $f^{-q}(D_{a_0})$ that contains $a$.
By choosing $\delta$ small, we may assume that
\begin{itemize}
    \item if $a \in f(U_i')$, then $D_a \subset f(U_i')$ and
    \item if $a \neq a_0$, then $f: D_a \to D_{f(a)}$ is a conformal isomorphism.
\end{itemize}
Let $U_0'',\dots, U_n''$ be the components of the open set
$$
\bigcup_{i=0}^m U_i' \cup \bigcup_{a \in f^{-l}(a_0)\setminus \{a_0,\dots, a_r\}} D_a.
$$
Then by construction, we have that
\begin{itemize}
    \item $\partial U_i'' \cap \mathbb{S}^1 \subset \{a_0,\dots, a_r\}$,
    \item $f\colon U_i'' \to f(U_i'')$ is conformal (because this is a proper map onto a subset of $V'$), and
    \item $f(U_i'')$ contains $U_j''$ whenever $U_j'' \cap f(U_i'') \neq \emptyset$.
\end{itemize}

Let $a_k \in \{a_0,\dots, a_r\}$ be a point not in the grand orbit of the parabolic fixed point $a_0$. Then $a_k$ is eventually mapped to a repelling periodic point. We have $a_k \in U''_i$ for some unique $i\in \{0,\dots,n\}$. We can inductively construct proper arcs $\mathcal{L}_k \subset U_i''$ passing through $a_k$ so that $\mathcal{L}_j \subset f(\mathcal{L}_k)$ if $f(a_k) = a_j$.
By cutting each $U''_i$ into finitely many pieces using these proper arcs $\mathcal{L}_k$, we obtain a collection of $r+1$ open neighborhoods $U_k$ of $A_k=\arc{(a_k,a_{k+1})}$, $k \in \{ 0, \dots , r\}$. 
Let $V_k= f(U_k)$. Then one can verify that \ref{condition:uv} is satisfied.
\end{proof}

\begin{proof}[Proof of Theorem \ref{theorem:intro_mating}]
    We follow closely the steps of the proof of \cite{LyubichMerenkovMukherjeeNtalampekos:David}*{Theorem 5.2}. Let $d$ be the degree of $f$ and $g$ and let $P(z)=z^d$.  By \ref{expansive:conjugate}, there exist homeomorphisms $h_1,h_2\colon \mathbb S^1\to \mathbb S^1$ that conjugate $P$ to $f,g$, respectively. Note that both $h_2$ and $h\circ h_1$ conjugate $P$ to $g$. By the uniqueness part in \ref{expansive:conjugate}, we may precompose $h_2$ with a rotation so that it agrees with $h\circ h_1$. See the diagram in Figure \ref{figure:mating}.

    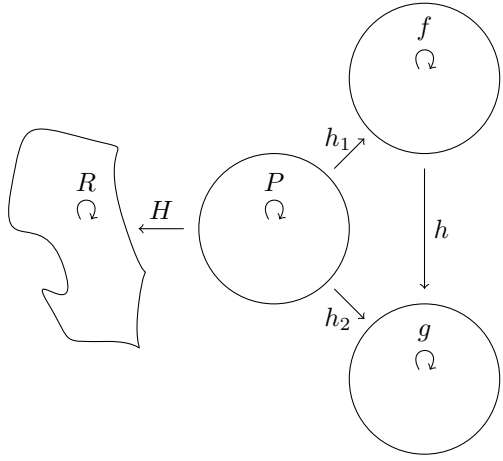
\begin{figure}
        \centering
        \begin{tikzpicture}
            \draw (0,0) circle (1cm) node (A) {};
            \draw (2,2) circle (1cm) node (B) {};
            \draw (2,-2) circle (1cm) node (C) {};

            \draw[->] (A) edge[loop above, out=120, in=60, looseness=7] node {$P$} (A);
            \draw[->] (B) edge[loop above, out=120, in=60, looseness=7] node {$f$} (B);
            \draw[->] (C) edge[loop above, out=120, in=60, looseness=7] node {$g$} (C);

            \draw[->] (0.8,0.8) -- node[pos=0.5, xshift=-0.15cm, yshift=0.2cm]{$h_1$} (1.2,1.2) ;
            \draw[->] (0.8,-0.8) -- node[pos=0.5, xshift=-0.15cm, yshift=-0.2cm]{$h_2$} (1.2,-1.2) ;
            \draw[->] (2,0.8)--node[pos=0.5, right]{$h$}(2,-0.8);

            \draw plot [smooth cycle, tension=2] coordinates{ (-3.5,0.5) (-2.5,1.2) (-2,0.1) (-1.8,-1) (-2.2,-1.5) (-3,-1) (-2.8,-0.5)};
            \node (F) at (-2.5,0) {};
            \draw[->] (F) edge[loop above, out=120, in=60, looseness=7] node {$R$} (F);      
            \draw[->] (-1.2,0)--node[pos=0.5, above]{$H$}(-1.8,0);
        \end{tikzpicture}
        \caption{The conjugacies in the proof of Theorem \ref{theorem:intro_mating}.}
        \label{figure:mating}
    \end{figure}

    We claim that $h_1$ has a David extension to the disk. Consider a Markov partition associated to $P$ that contains all fixed points. Then $h_1$ induces a Markov partition associated to $f$ that contains the parabolic fixed point, if any. Both Markov partitions of $P$ and $f$ satisfy conditions \ref{condition:hp}, \ref{condition:uv}, \ref{condition:qs}, and \ref{condition:qs_strong} by Theorem \ref{theorem:blaschke_conditions}. Moreover, each point of the Markov partition of $P$ is symmetrically hyperbolic. By Theorem \ref{theorem:extension_generalization}, $h_1$ has a David extension to $\bar \D$, which we still denote by $h_1$. With the same argument, $h_2$ has a David extension to $\widehat \C\setminus  \D$. 
    
    We define a Beltrami coefficient $\mu$ as follows. In $\D$ we let $\mu$ be the pullback of the standard complex structure under $h_1$. In $\widehat\C\setminus \bar \D$, $\mu$ is defined to be the pullback of the standard complex structure under $h_2$. By the David integrability theorem (Theorem \ref{theorem:integrability_david}), there exists a David homeomorphism $H$ of $\widehat \C$ with $\mu_H=\mu$. Consider the map
    \begin{align*}
        R=\begin{cases}
            H\circ h_1^{-1} \circ f \circ h_1\circ H^{-1} & \text{in}\,\, H(\D)\\
             H\circ h_2^{-1}\circ g \circ h_2\circ H^{-1}& \text{in}\,\, H(\widehat \C\setminus \D).
        \end{cases}
    \end{align*}
    Note that the two definitions agree on $J=H(\mathbb S^1)$ because $h$ conjugates $f$ to $g$ and $h_2= h\circ h_1$ on $\mathbb S^1$. 
    
    By Theorem \ref{theorem:stoilow}, $h_1\circ H^{-1}$ is a conformal map from $H(\D)$ onto $\D$. Let $\phi$ be the inverse of that map. Similarly, $h_2\circ H^{-1}$ is a conformal map from $H(\widehat\C\setminus \bar \D)$ onto $\widehat \C\setminus \overline \D$, and we let $\psi$ be its inverse. By Carath\'eodory's theorem, $\phi$ and $\psi$ extend homeomorphically to the closures of their domains and conjugate $f$ and $g$ to $R$, respectively.  Therefore, $R$ is holomorphic outside $H(\mathbb S^1)$. Moreover, $\phi=\psi\circ h$ on $\mathbb S^1$. 

    It remains to show that $R$ is holomorphic on $H(\mathbb S^1)$, which will imply that $R$ is rational. Note that $R$ is a branched cover with no critical points on $H(\mathbb S^1)$. Thus, each point of $H(\mathbb S^1)$ has a neighborhood in which $R$ is conformal. By Theorem \ref{theorem:removable}, $H(\mathbb S^1)$ is locally conformally removable, so $R$ is holomorphic on $H(\mathbb S^1)$.

    Finally, we show the ultimate uniqueness statement in Theorem \ref{theorem:intro_mating}. Suppose that there exist another Jordan curve $J_0$ with complementary regions $A_0,B_0\subset \widehat\C\setminus J_0$, a rational map $R_0$, and conformal maps $\phi_0\colon \overline \D\to \overline A_0$, $\psi\colon \overline \D\to \overline B_0$ such that $\phi_0$ conjugates $f$ to $ R_0$, $\psi_0$ conjugates $g$ to $R_0$, and $\phi=\psi\circ h$ on $\mathbb S^1$. We define 
    \begin{align*}
        G(z)=
        \begin{cases}
          \phi_0\circ \phi^{-1}(z), & z\in \overline A\\
          \psi_0\circ \psi^{-1}(z), & z\in \overline B
        \end{cases}
    \end{align*}
    and note that the two definitions agree on $J=\partial A=\partial B$. The map $G$ is a homeomorphism of $\widehat \C$ that is conformal in $\widehat \C\setminus J$ and conjugates $R$ to $R_0$. By Theorem \ref{theorem:removable}, $J$ is conformally removable, so $G$ is conformal in $\widehat \C$, and thus it is a M\"obius transformation. 
\end{proof}

\bibliography{biblio}

\begin{bibdiv}
\begin{biblist}

\bib{AstalaIwaniecMartin:quasiconformal}{book}{
      author={Astala, Kari},
      author={Iwaniec, Tadeusz},
      author={Martin, Gaven},
       title={Elliptic partial differential equations and quasiconformal
  mappings in the plane},
      series={Princeton Mathematical Series},
   publisher={Princeton University Press, Princeton, NJ},
        date={2009},
      volume={48},
}

\bib{BeurlingAhlfors:extension}{article}{
      author={Beurling, A.},
      author={Ahlfors, L.},
       title={The boundary correspondence under quasiconformal mappings},
        date={1956},
     journal={Acta Math.},
      volume={96},
       pages={125\ndash 142},
}

\bib{BeardonMinda}{incollection}{
      author={Beardon, A.~F.},
      author={Minda, D.},
       title={The hyperbolic metric and geometric function theory},
        date={2007},
   booktitle={Quasiconformal mappings and their applications},
   publisher={Narosa, New Delhi},
       pages={9\ndash 56},
}

\bib{BeardonPommerenke}{article}{
      author={Beardon, A.~F.},
      author={Pommerenke, Ch.},
       title={The {P}oincar\'{e} metric of plane domains},
        date={1978},
     journal={J. London Math. Soc. (2)},
      volume={18},
      number={3},
       pages={475\ndash 483},
}

\bib{ChenChenHe:boundary}{article}{
      author={Chen, Jixiu},
      author={Chen, Zhiguo},
      author={He, Chengqi},
       title={Boundary correspondence under {$\mu(z)$}-homeomorphisms},
        date={1996},
     journal={Michigan Math. J.},
      volume={43},
      number={2},
       pages={211\ndash 220},
}

\bib{CJY}{article}{
      author={Carleson, Lennart},
      author={Jones, Peter~W.},
      author={Yoccoz, Jean-Christophe},
       title={Julia and {J}ohn},
        date={1994},
     journal={Bol. Soc. Brasil. Mat. (N.S.)},
      volume={25},
      number={1},
       pages={1\ndash 30},
}

\bib{CovenReddy:expansive}{incollection}{
      author={Coven, Ethan~M.},
      author={Reddy, William~L.},
       title={Positively expansive maps of compact manifolds},
        date={1980},
   booktitle={Global theory of dynamical systems ({P}roc. {I}nternat. {C}onf.,
  {N}orthwestern {U}niv., {E}vanston, {I}ll., 1979)},
      series={Lecture Notes in Mathematics},
      volume={819},
   publisher={Springer, Berlin},
       pages={96\ndash 110},
}

\bib{Federer:gmt}{book}{
      author={Federer, Herbert},
       title={Geometric measure theory},
      series={Grundlehren der mathematischen Wissenschaften},
   publisher={Springer-Verlag New York Inc., New York},
        date={1969},
      volume={153},
}

\bib{Heinonen:metric}{book}{
      author={Heinonen, Juha},
       title={Lectures on analysis on metric spaces},
      series={Universitext},
   publisher={Springer-Verlag, New York},
        date={2001},
}

\bib{HemmingsenReddy:expansive}{article}{
      author={Hemmingsen, Erik},
      author={Reddy, William},
       title={Expansive homeomorphisms on homogeneous spaces},
        date={1969},
     journal={Fund. Math.},
      volume={64},
       pages={203\ndash 207},
}

\bib{HT04}{article}{
      author={Ha\"issinsky, Peter},
      author={Tan, Lei},
       title={Convergence of pinching deformations and matings of geometrically
  finite polynomials},
        date={2004},
     journal={Fund. Math.},
      volume={181},
      number={2},
       pages={143\ndash 188},
}

\bib{IOZ21}{article}{
      author={Iwaniec, Tadeusz},
      author={Onninen, Jani},
      author={Zhu, Zheng},
       title={Singularities in $\mathscr{L}^p$-quasidisks},
        date={2021},
     journal={Ann. Fenn. Math.},
      volume={46},
      number={2},
       pages={1053\ndash 1069},
}

\bib{KarafylliaNtalampekos:extension}{article}{
      author={Karafyllia, Christina},
      author={Ntalampekos, Dimitrios},
       title={Extension of boundary homeomorphisms to mappings of finite
  distortion},
        date={2022},
     journal={Proc. London Math. Soc.},
      volume={125},
      number={3},
       pages={488\ndash 510},
}

\bib{LyubichMerenkovMukherjeeNtalampekos:David}{article}{
      author={Lyubich, Mikhail},
      author={Merenkov, Sergei},
      author={Mukherjee, Sabyasachu},
      author={Ntalampekos, Dimitrios},
       title={David extension of circle homeomorphisms, welding, mating, and
  removability},
        date={2023},
     journal={Mem. Amer. Math. Soc.},
       pages={to appear},
}

\bib{LN24b}{article}{
      author={Luo, Yusheng},
      author={Ntalampekos, Dimitrios},
       title={Uniformization of gasket {J}ulia sets},
        date={2024},
        note={Preprint},
}

\bib{Luo22}{article}{
      author={Luo, Yusheng},
       title={On geometrically finite degenerations {II}: convergence and
  divergence},
        date={2022},
     journal={Trans. Amer. Math. Soc.},
      volume={375},
      number={5},
       pages={3469\ndash 3527},
}

\bib{Luo23}{article}{
      author={Luo, Yusheng},
       title={On geometrically finite degenerations {I}: boundaries of main
  hyperbolic components},
        date={2024},
     journal={J. Eur. Math. Soc.},
      volume={26},
      number={10},
       pages={3793\ndash 3861},
}

\bib{LehtoVirtanen:quasiconformal}{book}{
      author={Lehto, O.},
      author={Virtanen, K.~I.},
       title={Quasiconformal mappings in the plane},
     edition={Second edition},
      series={Grundlehren der mathematischen Wissenschaften},
   publisher={Springer-Verlag, New York-Heidelberg},
        date={1973},
      volume={126},
}

\bib{McM94}{book}{
      author={McMullen, Curtis},
       title={Complex dynamics and renormalization},
      series={Annals of Mathematics Studies},
   publisher={Princeton University Press, Princeton, NJ},
        date={1994},
      volume={135},
}

\bib{PrzytyckiUrbanski:book}{book}{
      author={Przytycki, Feliks},
      author={Urba\'{n}ski, Mariusz},
       title={Conformal fractals: {E}rgodic theory methods},
      series={London Mathematical Society Lecture Note Series},
   publisher={Cambridge University Press, Cambridge},
        date={2010},
      volume={371},
}

\bib{Shi00}{incollection}{
      author={Shishikura, Mitsuhiro},
       title={Bifurcation of parabolic fixed points},
        date={2000},
   booktitle={The {M}andelbrot set, theme and variations},
      series={London Mathematical Society Lecture Note Series},
      volume={274},
   publisher={Cambridge University Press, Cambridge},
       pages={325\ndash 363},
}

\bib{Zakeri:boundary}{article}{
      author={Zakeri, Saeed},
       title={On boundary homeomorphisms of trans-quasiconformal maps of the
  disk},
        date={2008},
     journal={Ann. Acad. Sci. Fenn. Math.},
      volume={33},
      number={1},
       pages={241\ndash 260},
}

\end{biblist}
\end{bibdiv}
\end{document}